\documentclass{article}
\usepackage[hmargin=2.5cm,vmargin=3cm,bindingoffset=0.5cm]{geometry}

\usepackage{fancyhdr,graphics,graphicx,xy}
\usepackage{ifthen}
\usepackage{amssymb,amsmath,mathtools,enumerate,tikz-cd}
\usepackage{amsthm}
\usepackage{graphics}
\usepackage{dsfont}
\usepackage{algpseudocode}
\usepackage{adjustbox}
\usepackage{bbm}
\usepackage{tikz}
\usepackage[affil-it]{authblk}

\usetikzlibrary{arrows}
\usetikzlibrary{snakes}
\graphicspath{/Users/huubdejong/Downloads/}

\def\N{\mathbb{N}}

\def\R{\mathbb{R}}

\def\Z{\mathbb{Z}}

\def\limn{\lim_{n\to\infty}}

\def\supp{\operatorname{supp}}

\title{On Sets of Periodic Orbit Lengths in Finitely Presented Dynamical Systems}
\author{Huub de Jong%
  \thanks{Electronic address: \texttt{huub.dejong@math.ubc.ca}}}
\affil{The University of British Columbia, Vancouver}
\date{\today}

\newtheorem{theorem}{Theorem}[section]
\newtheorem{corollary}[theorem]{Corollary}
\newtheorem{definition}[theorem]{Definition}
\newtheorem{lemma}[theorem]{Lemma}

\newtheorem{proposition}[theorem]{Proposition}
\newtheorem{remark}[theorem]{Remark}

\begin{document}

\maketitle

\begin{abstract}
    We classify the sets of natural numbers $n$ for which certain dynamical systems $(X,f)$ on a compact metric space $X$ have a periodic point of (least) period $n$. Interest in this question dates back to Sharkovskii's theorem for continuous maps on intervals of the real line, but it also ties to checkable conditions for Krieger's embedding theorem for symbolic dynamical systems. Given a system for which the logarithmic derivative of the Artin-Mazur zeta function is rational, we use the Skolem-Mahler-Lech theorem to classify for which $n$ the system has a periodic point of (not necessarily least) period $n$. Moreover, we build on work on finitely presented (FP) systems and their relationship to symbolic dynamics to classify the set of least periods, that is periodic orbit lengths, for arbitrary FP systems, extending a known classification for shifts of finite type. We also provide several constructions to realize any such least period sets.
\end{abstract}

\section*{Acknowledgements} 

This paper is motivated by earlier work of Madeline Doering and Ronnie Pavlov, and insightful discussions with the latter strengthened several of the results presented here. The author thanks Brian Marcus for his guidance and extensive feedback. The proofs of Theorems \ref{thm-realizeLPSsofic} and \ref{thm-realizeLPSgap} owe a great deal to Chengyu Wu, who also provided a many detailed comments on the work as a whole. Finally, the author expresses his gratitude to the anonymous referee, who provided several detailed comments that helped polish this paper into its final form. 

\section{Introduction}\label{intro}

Let $X$ be a nonempty compact metric space, and $f:X\to X$ a continuous surjective map. Moreover, denote $\N:=\{1,2,3,...\}$. Our primary interest in this paper is in the collection of \textit{periodic points} of period $n\in\N$, denoted
$$P_n(f):=\{x\in X:f^n(x)=x\},$$
and the periodic points of \textit{least period} $n$, which we denote
$$Q_n(f):=\{x\in X:f^n(x)=x,\text{ and }f^k(x)\neq x\text{ for }1\leq k<n\}.$$
For convenience, we will also write $p_n(f):=|P_n(f)|$, and $q_n(f):=|Q_n(f)|.$
It is clear that $$p_n(f)=\sum_{k|n} q_k(f). $$
In particular, we are interested in characterizing the \textit{period set} of $f$, $$\operatorname{PS}(f):=\supp(p_n(f))=\{n\in\N:p_n(f)>0\},$$
and the \textit{least period set}, that is the set of periodic orbit lengths, of $f$
$$\operatorname{LPS}(f):=\supp(q_n(f))=\{n\in\N:q_n(f)>0\}.$$ 
This question is of interest in various situations, two of which we will briefly explain.

First, in 1964, Sharkovskii completely classified $\operatorname{LPS}(f)$ when $f$ is a continuous selfmap of a (possibly infinite) interval $I\subset \R$. In his proof, he first constructs what is sometimes known as the \textit{Sharkovskii ordering} $\triangleleft$ on $\N$, given by \begin{align*}
\begin{matrix}
    &3 & \triangleleft &5& \triangleleft &7 &\triangleleft  & 9 &\triangleleft&\cdots \\
     \triangleleft&2\cdot 3&\triangleleft &2\cdot 5& \triangleleft &2\cdot 7 &\triangleleft & 2\cdot 9 &\triangleleft &\cdots \\
     \triangleleft& 2^2\cdot 3&\triangleleft &2^2 \cdot 5 &\triangleleft &2^2\cdot 7 &\triangleleft & 2^2\cdot 9 &\triangleleft &\cdots \\
      &\vdots & & \vdots & & \vdots & & \vdots & \\
     &\cdots &\triangleleft & 2^3&\triangleleft &2^2 &\triangleleft &2 &\triangleleft &1   
     \end{matrix}
\end{align*}
A \textit{tail} in a total ordering is a set $S$ such that whenever $n\in S$ and $n \triangleleft m$, then $m\in S$. The result known as
\textit{Sharkovskii's Theorem}  then states the following.
\begin{theorem}[Sharkovskii \cite{sharkovskyCoexistenceCyclesContinuous2024}]\label{thm-sharkovskii} A set $S\subset \N$ is the LPS of a continuous selfmap $f$ on an interval $I\subset \R$ if and only if $S$ is a tail in the Sharkovskii ordering.
\end{theorem}

Our second motivating example comes from symbolic dynamics. This field is rich, but we will only introduce some basic tools. For an extensive introduction, we refer to \cite{Lind_Marcus_1995}. Let $\mathcal{A}$ be a finite set of symbols equipped with the discrete topology, called the \textit{alphabet}. We call $\mathcal{A}^\Z$ the \textit{full shift on $\mathcal{A}$}, and equip it with the \textit{shift map} $\sigma$ defined by
$$\sigma(x)_i=x_{i+1}.$$
A \textit{shift space} is then a subset of $\mathcal{A}^\Z$ that is $\sigma$-invariant and compact with respect to the product topology. We will denote shift spaces $X,Y$, etc. and write $p_n(X)$ as opposed to $p_n(\sigma)$, and use the same convention for other notation introduced throughout. 

A \textit{word} $w$ on $\mathcal{A}$ is a finite sequence of symbols in $\mathcal{A}$ whose length will be denoted $|w|$. We write $w^\infty$ for the point $x=\cdots w.ww\cdots$, where the full stop indicates that $x_0=w_0,$ the first symbol in $w$. A word $w$ is \textit{purely periodic} if there exists a word $v$ and integer $k>1$ such that $w=v^k$.

More generally, for $a,b\in\Z,a\leq b$ we write $x_{[a,b]}:=x_ax_{a+1}\cdots x_{b-1}x_b$. The \textit{cylinder set} of a word $w$ will be $$C_w=\{x\in X: x_{[0,|w|-1]}=w \}.$$
We say that a shift space $X$ is a \textit{shift of finite type} (SFT) if there exists a \textit{finite} collection of words $\{w_1,...,w_n\}$ on $\mathcal{A}$ such that 
$$\mathcal{A}^\Z\setminus X = \bigcup_{j=1}^n \bigcup_{i\in\Z} \sigma^i(C_{w_j}).$$
Stated otherwise, an SFT is a shift space that can be described by ``forbidding'' finitely many words from occurring in any point in the shift space. 

By recoding to what is known as a \textit{higher block representation} of $X$ we may always assume that the forbidden words $w_j$ are of length 2, in which case we will say that $X$ is a \textit{one-step} SFT. This will be a standing assumption on SFTs in what follows. 

One of the main benefits of this assumption is that, as per Theorem 2.3.2 in \cite{Lind_Marcus_1995}, the SFT can be presented by the collection of bi-infinite walks, encoded as sequences of edges, on some finite graph $G$. If $G$ is such a graph with adjacency matrix $A$, we will write $X_G$ or $X_A$ to mean the associated SFT and $\sigma_G$ or $\sigma_A$ for the corresponding shift map. The \textit{global period} of an SFT $X$ is defined to be $\gcd(\operatorname{PS}(X))$, which is equal to the period of any graph presentation of $X$.

A general dynamical system $(Y,f)$ is a \textit{factor} of an SFT $X$ if there exists a continuous surjection, called a \textit{factor map}, $\phi:X\to Y$ such that the diagram

$$\begin{tikzcd}
X \arrow[r, "\sigma"] \arrow[d, "\phi"] & X \arrow[d, "\phi"] \\
Y \arrow[r, "f"]                        & Y                  
\end{tikzcd}$$
commutes.\footnote{The commutative diagrams used throughout were created with the help of https://github.com/yishn/tikzcd-editor, the creator of which we thank.} If there exists an $M\in \N$ such that $|\phi^{-1}(y)|\leq M$ for all $y\in Y$, we say $\phi$ is \textit{bounded-to-one}. When there is no confusion, we will use ``$(Y,f)$ is a factor of $X$'' and ``$Y$ is a factor of $X$'' interchangeably. If a factor map $\phi$ is a homeomorphism, we call $\phi$ a \textit{conjugacy}. 

If $Y$ as above is both a factor of an SFT and a shift space itself, we say that $Y$ is \textit{sofic} \cite{weissSubshiftsFiniteType1973}. In parallel with the graph presentation for SFTs, sofic shifts can be presented by bi-infinite walks on a graph $G$ modulo a labeling $\Phi$ of the edges in $G$. We will call $(G,\Phi)$ a \textit{presentation} of a sofic shift.

Suppose $X$ is a compact metric space. We say a homeomorphism $f:X\to X $ is \textit{expansive} if there exists a constant $\varepsilon>0$ such that for any $x\neq y\in \Omega$, there exists an $n\in\Z$ such that $$d(f^nx,f^ny)>\varepsilon.$$ We call $\varepsilon$ an \textit{expansive constant} for $f$.

If a dynamical system $(Y,f)$ is an expansive factor of an SFT, then we say that it is \textit{finitely presented} (FP). Standard examples of FP systems are sofic shifts and \textit{hyperbolic toral automorphisms}, i.e. diffeomorphisms on $\R^d/\Z^d$ induced by a matrix $A\in \operatorname{GL}_d(\Z)$ with no eigenvalues of modulus 1 \cite{AdlerWeissMarkovPartitions,SinaiMarkovPartitions}. We will discuss the definition further in Section \ref{sec-fp}.

A general dynamical system $(X,f)$ is \textit{topologically transitive} if for every pair of nonempty open sets $U,V\subset X$ there exists an $n>0$ such that 
$$ U\cap f^{-n}(V) \neq \emptyset.$$
In particular, if for all nonempty open $U,V\subset X$, the above is true for all $n$ larger than some $N=N(U,V)>0$, we say that $(X,f)$ is \textit{topologically mixing.}

Let $X$ be an SFT represented by a graph $G$ with adjacency matrix $A_G$. We say that $X$ is \textit{irreducible}, respectively \textit{mixing}, if $A_G$ is irreducible, respectively \textit{primitive}. Note that irreducibility of a shift space is thus equivalent to being topologically transitive, and mixing  means topologically mixing.

The interest in the LPS of shift spaces comes from the famed \textit{Krieger embedding theorem}. Here, $h(X)$ denotes the \textit{entropy} of $X$, which will be defined in Section 4.
\begin{theorem}[\cite{Krieger_1982}]\label{thm-krieger}
    Let $Y$ be a mixing SFT and $X$ any shift space such that $h(X)<h(Y)$. Then $X$ embeds into $Y$ if and only if $q_n(X)\leq q_n(Y)$ for all $n\in\N$. 
\end{theorem}
A similar statement is true when $X$ is irreducible, see work by Marcus, Meyerovitch, Thomsen, and Wu \cite{MarcusEtAlEmbeddingTheorems}. We point out the irreducible case because here the difference in LPS between SFTs and sofic shifts is more pronounced, see Table \ref{tab:results}. We will touch on this again in Section \ref{realizations}. 

The \textit{zeta function} of a dynamical system $(X,f)$ is $$\zeta_f(t):=\exp\left(\sum_{n\in\N} \frac{p_n(f)}{n}t^n\right),$$
wherever it is defined. In this definition we assume that $p_n(f)<\infty$ for all $n$, which will be the case for all systems under consideration. 

For two integers $d,k\in \N$ we will denote $$d(\N+k):=\{d(k+1),d(k+2),d(k+3),...\}.$$
More generally, for any set $S=\{s_1,s_2,..\}\subset \N$ and integer $d\in\N$ we will write $$dS=\{ds_1,ds_2,...\}.$$

We are now ready to state our main results. 

\begin{theorem}\label{thm-PSofrationalzeta}
    Let $(X,f)$ be a dynamical system with zeta function $\zeta_f(t)$ with positive radius of convergence, such that $\frac{d}{dt}\ln\zeta_f(t)$ is rational. Then the period set of $X$ is of the form
$$
\operatorname{PS}(X)=\bigcup_{i=0}^nd_{i}(\mathbb{N}+k_{i}),
$$
where the union is possibly empty. In particular, if $\zeta_f(t)$ is rational, the conclusion holds.
\end{theorem}

A set $S\subset\N$ is \textit{cofinite} if its complement in $\N$ is finite. The following result classifies the LPS for FP systems. The converse is also true: see Theorem \ref{thm-realizeLPSsofic}, where we construct a sofic shift (i.e. an FP system) that realizes an LPS of the form below. 

\begin{theorem}\label{thm-LPSofFP-intro}
    Let $(\Omega,f)$ be an FP system. Then $\operatorname{LPS}(f)$ is of the form $$F\cup \bigcup_{i=1}^n d_iS_i,$$
    where $F$ is finite, each $S_i$ is cofinite, and either $F$ or the big union may be empty, but not both.

    Moreover, if $(\Omega,f)$ is a topologically transitive FP system, then $\operatorname{LPS}(f)$ is either a singleton or of the form above with $n\geq 1$, i.e. at least one cofinite set $S_i$.
    
    Finally, if $(\Omega,f)$ is mixing, then $\operatorname{LPS}(f)$ is either $\{1\}$ or a cofinite set.
\end{theorem}

If the FP system in theorem above is a mixing sofic shift, the result was shown by Doering in \cite{DoeringMasters}. The classification for SFTs can be made more specific. The following theorem was proved directly by Doering and Pavlov; we provide an alternative proof as a special case of the result above.

\begin{theorem}[\cite{DoeringPavlov}]\label{thm-LPSofirredSFT-intro}
    Let $X$ be an SFT. Then $\operatorname{LPS}(X)$ is of the form $$F\cup \bigcup_{i=1}^n d_iS_i,$$
    where $F$ is finite, each $S_i$ is cofinite, and either $F$ or the big union may be empty, but not both.
    
    If $X$ is irreducible, then $\operatorname{LPS}(X)$ is either a singleton $\{d\}$ or of the form $dS$ for $S$ cofinite, and if $X$ is mixing, then $\operatorname{LPS}(X)$ is either $\{1\}$ or cofinite.
\end{theorem}

The results on LPSs are summarized handily in Table \ref{tab:results}, where any set denoted $S$ or $S_i$ is understood to be cofinite and $F$ denotes a finite set. The condition $(*)$ indicates that if the LPS of a toplogically transitive FP system or of an irreducible sofic shift is finite, it must be a singleton. This condition does not hold in the general case.

\begin{table}
\begin{center}
\begin{tabular}{|c||c|c|c|}
    \hline
     & SFT & Sofic & FP \\ 
 \hline\hline
    Mixing & $\{1\}$ or $S$ & $\{1\}$ or $S$ & $\{1\}$ or $S$  \\ 
    \hline
    Transitive & $\{d\}$ or $dS$ & $F\cup \bigcup_{i=1}^n d_iS_i \;\;{(*)}$ & $F\cup \bigcup_{i=1}^n d_iS_i \;\;{(*)}$\\
    \hline General & $F\cup \bigcup_{i=1}^n d_iS_i$ & $F\cup \bigcup_{i=1}^n d_iS_i$ & $F\cup \bigcup_{i=1}^n d_iS_i$  \\ \hline
\end{tabular}
\caption{
      \label{tab:results}
      Summary of classification of least period sets.}
\end{center}
\end{table}

The structure of this paper is as follows. In Section \ref{PS} we prove Theorem \ref{thm-PSofrationalzeta}, and in Section \ref{LPS} we prove Theorems \ref{thm-LPSofFP-intro} and \ref{thm-LPSofirredSFT-intro}. The latter section also contains some intermediate results that may be of independent interest, and we further provide a classification for the LPS of shifts known as \textit{gap shifts} by elementary tools.

Section \ref{realizations} is dedicated to constructions that show that any LPS identified in Section \ref{LPS} can indeed be realized by some dynamical system, in fact a shift space.

\section{Period Sets}\label{PS}

In this section, we will prove Theorem \ref{thm-PSofrationalzeta}, which in particular classifies the period set of dynamical systems with a rational zeta function. Systems with this property include rational maps of the complex plane \cite{Hinkkanen_1994} and FP Systems \cite{FriedFPSystems}, which include Axiom A diffeomorphisms \cite{ManningRationalZeta} and SFTs \cite{Lind_Marcus_1995}, fall in this category.

Recall that we are trying to show that if $(X,f)$ is a dynamical system whose zeta function with positive radius of convergence such that $\frac{d}{dt}\ln\zeta_f(t)$ is rational, then the period set of $X$ is of the form
$$
\operatorname{PS}(X)=\bigcup_{i=0}^nd_{i}(\mathbb{N}+k_{i}),
$$
where the union is possibly empty.

We will require a lemma due to Schützenberger, which we modify slightly to suit our purposes. A sequence of complex numbers $p_n$ satisfies a \textit{linear recurrence (of order $k$)} if there exist a finite collection of complex numbers $b_i$ for $0\leq i \leq k$, not all 0, such that for all $n$ larger than some $N\in \N$, $$\sum_{i=0}^k b_ip_{n-i}=0.$$ The proof of the lemma is fairly elementary and we include it here for completeness. 
\begin{lemma}[Schützenberger]\label{lem-schutzenberger}
    Consider a sequence $(p_n)_{n=0}^\infty$ of complex numbers such that on some open disk in $\mathbb{C}$, centered at 0 we have
$$
\sum_{n=0}^\infty p_{n}t^n = \frac{A(t)}{B(t)},
$$
with $A(t),B(t)\in\mathbb{C}[t]$. Then for large enough $n$, $(p_n)$ satisfies a linear recurrence. Moreover, if $p_n\in \mathbb{Q}$ for all $n$, then the coefficients of $A(t),B(t)$ are rational.
\end{lemma}

\begin{proof}
    Since the series does not have a pole at 0, we may write 
$$\sum_{n=0}^\infty p_{n}t^n= \frac{a_{k}t^k+a_{k-1}t^{k-1}+ \dots +a_{0} }{b_{m}t^m+\dots + 1},$$
so that the numerator and denominator are relatively prime in $\mathbb{C}[t]$. Then we have 
\begin{align}\label{recurrence}
    ({b_{m}t^m+\dots +1})\left(\sum_{n=0}^\infty p_{n}t^n\right)={a_{k}t^k+a_{k-1}t^{k-1}+ \dots +a_{0} },
\end{align}
and so for $n> k,$ 
$$
p_{n}=\sum_{i=1}^m -p_{n-i}b_{i}.
$$
The linear recurrence above is in fact of minimal order, since otherwise $\operatorname{gcd}(A(t),B(t))\neq 1$. Let now
$$
P_{n}=\begin{pmatrix}
p_{n} & \dots & p_{n-m+1} \\
\vdots & \ddots & \vdots  \\
p_{n-m+1}& \dots & p_{n-2m+2}
\end{pmatrix}
$$
so that
$$
P_{n+1} =
\begin{pmatrix}
{-b_{1}} & {-b_{2}} & \dots & {-b_{m}} \\
1 & 0 & \dots & 0 \\
\vdots  & \ddots & & \vdots \\
0 & \dots & 1 & 0
\end{pmatrix} 
P_{n}.
$$
Since $(p_n)$ does not satisfy any shorter linear recurrence, $P_n$ is invertible. It follows that
$$
P_{n+1}P_{n}^{-1} =
\begin{pmatrix}
{-b_{1}} & {-b_{2}} & \dots & {-b_{m}} \\
1 & 0 & \dots & 0 \\
\vdots  & \ddots & & \vdots \\
0 & \dots & 1 & 0
\end{pmatrix},
$$
hence $B(t)\in \mathbb{Q}[t]$. By (\ref{recurrence}), $A(t)\in \mathbb{Q}[t]$, and the proof is complete.
\end{proof}

The second part of the lemma above implies that if a system has a rational zeta function, its poles and zeroes must be algebraic numbers.

We will also use the following powerful theorem, which is much broader than our use case. We refer the reader to \cite{HanselSMLTheorem} for a comprehensive proof. For a sequence $(a_n)$, we denote its \textit{zero set} by $$Z(a_n):=\{n\in\N:a_n=0\}.$$

\begin{theorem}[Skolem-Mahler-Lech]\label{thm-SML}
    Let $(a_n)$ be a sequence of complex numbers that satisfies a linear recurrence. Then there exist $r,b_1,\dots,b_k$ with $k$ possibly $0$ such that the zero set of $(a_{n})$ is
$$
Z(a_{n})=F \cup \bigcup_{i=1}^k (b_{i}+r\mathbb{N})
$$
where F is finite.
\end{theorem}

We have now equipped ourselves with all the tools required to prove our main result of this section.

\begin{proof}[Proof of Theorem \ref{thm-PSofrationalzeta}.]
    Let $\zeta(t)$ denote the dynamical zeta function of $(X,f)$. By definition of the zeta function, we have 
\begin{align*}
   \frac{d}{dt}\ln(\zeta(t))=\frac{d}{dt}\sum_{n=1}^\infty \frac{p_{n}}{n}t^n=\sum_{n=1}^\infty p_{n} t^{n-1}.
\end{align*}
Since $\frac{d}{dt}\ln\zeta(t)$ is assumed to be rational, by Lemma \ref{lem-schutzenberger}, the sequence $(p_n)$ satisfies a linear recurrence. By Theorem \ref{thm-SML}, this means that for some $r,K\in \mathbb{N}$ and some $B\subset \{0,1,\dots,r-1\}$ and finite $F\subset\mathbb{N}$,
$$
Z(p_{n})=F \cup \bigcup_{i \in B}\{ m = i \mod r, m\geq K\},
$$
and thus 
\begin{equation}\label{PSisZcomplement}
\operatorname{PS}(X)=Z(p_{n})^c=F'\cup \bigcup_{i\in B^c}\{ m=i \mod r, m\geq K \},
\end{equation}
where $F'=\{ 1,\dots,K-1 \}\setminus F$.

We shall write $i|j \mod r$ to mean that for some $a\in\mathbb{N}$, $ai=j\mod r$, or equivalently that $j$ is an integer combination of $i$ and $r$. We note further that by definition of $(p_n)$, for all $k,n\in \mathbb{N}$, 
\begin{align}\label{monotonicity}
    n\in \operatorname{PS}(X)\implies kn\in \operatorname{PS}(X).
\end{align}
This means that if $i\in B^c$, and $i|j \mod r$ , then $j\in B^c$. Indeed, suppose $K\leq i+kr\in \operatorname{PS}(X)$ and $a(i+kr)=j+br$ for some constants $a,b,$ and $k$. Then $K\leq j+br\in \operatorname{PS}(X)$, so $j\in B^c$.

We now show that 
\begin{align}\label{divisors}
    \bigcup_{i\in B^c}\{ m=i \mod r \}=\bigcup_{i\in B^c} i\mathbb{N}.
\end{align}

On the one hand, for fixed  $n\in \mathbb{N}$, we have $i|ni \mod r$. Hence $ni =j+kr$ for some $j\in B^c$, implying $$\bigcup_{i\in B^c}\{ m=i \mod r \}\supseteq\bigcup_{i\in B^c} i\mathbb{N}.$$

On the other hand, if $m=i+kr$, we can write $m=k'\gcd(i,r)$ for some $k'\in\mathbb{N}$. Bézout's Lemma, Theorem 2.10 in \cite{judsonAbstractAlgebraTheory2024}, tells us that $i|\gcd(i,r)\mod r$, and so  $\gcd(i,r)\in B^c$. This means $$\bigcup_{i\in B^c}\{ m=i \mod r \}\subseteq\bigcup_{i\in B^c} i\mathbb{N},$$ establishing (\ref{divisors}).

Applying this equality to (\ref{PSisZcomplement}), we get that

\begin{align*}
\operatorname{PS}(X)&= F'\cup \bigcup_{i\in B^c}\{ m=i \mod r, m\geq K \} \\ 
&= F'\cup \bigcup_{i\in B^c}\{ m=i \mod r\}\cap\{m\geq K \} \\ \\
&\stackrel{(\ref{divisors})}{=} F'\cup \bigcup_{i\in B^c} i\mathbb{N} \cap \{ m\geq K \}\\
&\stackrel{(\ref{monotonicity})}{=}\bigcup_{d\in F'}d\mathbb{N} \cup\bigcup_{i\in B^{c}} i\left( \mathbb{N}+\left\lfloor  \frac{K}{i}  \right\rfloor  \right)  \\
\end{align*} Both $F'$ and $B^c$ are finite, giving the desired expression.

Finally, if $\zeta(t)=\frac{R(t)}{S(t)}$ for some polynomials $R,S\in \mathbb{Q}[t]$, then $$\frac{d}{dt}\ln\zeta(t)=\frac{R'(t)}{R(t)}-\frac{S'(t)}{S(t)}.$$ This expression is rational, so $\zeta(t)$ satisfies the assumptions of the theorem, and the conclusion holds. 
\end{proof}

\section{Least Period Sets}\label{LPS}

\subsection{Entropy and Finitely Presented  Systems}\label{sec-fp}
For our result on least period sets, we will need some tools from topological dynamics, which we develop here. For a detailed introduction, we refer to \cite{WaltersIntroductionErgodicTheory2000}.

Let $(X,d)$ be a compact metric space and $f:X\to X$ our transformation of interest. The transformation $f$ induces a new metric
$$d_n(x,y)=\max_{0\leq i < n} d(f^ix,f^iy)$$
for each $n\in \mathbb{N}$. We say that a set of points $E$ is $(n,\varepsilon)$-\textit{separated} if for $x,y\in E$, $x\neq y$, we have $d_n(x,y)>\varepsilon$. 
Let $$s_n(\varepsilon):=\max_{E\subset X} \{|E|: E \text{ is $(n,\varepsilon)$-separated}\},$$
which is finite for fixed $n,\varepsilon$ by compactness of $X$. 

We can now state one way of defining the topological entropy of a dynamical system. This particular definition is due to Bowen, see Chapter 7.2 in \cite{WaltersIntroductionErgodicTheory2000}.  

\begin{definition}
    The topological entropy of a dynamical system $(X,f)$ as above is $$h(f)=\lim_{\varepsilon \searrow 0} \limsup_{n\to\infty} \frac{1}{n} \log s_n(\varepsilon).$$
\end{definition}

The first half of the next lemma is well-known and follows for example from combining Theorem 7.2 and Theorem 7.8 in  \cite{WaltersIntroductionErgodicTheory2000}. A purely topological proof of the second half can be found in chapter IX, Theorem 1.8 of \cite{robinson1995dynamical}, which uses Theorem 17 from \cite{BowenEntropyGroupEndo}.

\begin{lemma}\label{lem-bdtoonepreservesentropy}
    If $\pi:(X,f)\to (Y,g)$ is a factor map between dynamical systems, then 
    $$h(g)\leq h(f).$$ Moreover, if $\pi$ is also bounded-to-one, then $$h(g)=h(f).$$
\end{lemma}

We will also use a property called \textit{entropy-minimality}, which says that any proper subsystem must have strictly smaller entropy than the whole system. 

\begin{definition}
    A dynamical system $(X,f)$ is entropy-minimal if, whenever $Y$ is a closed invariant proper subset of $X$, we have $$h(f|_Y)<h(f).$$
\end{definition}

Irreducible SFTs, for example, have this property. The following lemma is Theorem 4.4.7 in \cite{Lind_Marcus_1995}.

\begin{lemma}\label{lem-SFTentropyminimal}
    Suppose $(X,\sigma)$ is an irreducible SFT. Then $(X,\sigma)$ is entropy-minimal.
\end{lemma}

Combining this result with Lemma \ref{lem-bdtoonepreservesentropy} immediately shows that entropy-minimality holds for any bounded-to-one factor of an irreducible SFT. 

\begin{corollary}\label{cor-FPentropyminimal}
    Suppose $\pi:X\to\Omega$ is a bounded-to-one factor map from an irreducible SFT $X$ to some dynamical system $(\Omega,f)$. Then $(\Omega,f)$ is entropy-minimal. In particular, topologically transitive FP systems are entropy-minimal.
\end{corollary}

\begin{proof}
    Let $Z\subsetneq\Omega$ be a proper subsystem. Since $Z$ is closed and $f$-invariant, $Y:=\pi^{-1}(Z)$ is compact and a strict subshift of $X$. Moreover, since $\pi$ is bounded-to-one, so is $\pi':=\pi|_Y$, giving the following diagram: $$\begin{tikzcd}
Y \arrow[d, "\pi'"] \arrow[r, hook] & X \arrow[d, "\pi"] \\
Z \arrow[r, hook]                                           & \Omega            
\end{tikzcd}$$ 
    Thus,
    $$h(f|_Z) = h(Y) <h(X)=h(f).$$
\end{proof}

Finally, the following well-known lemma on entropy for expansive homeomorphisms is in \cite{bowen1978axiom}.

\begin{lemma}\label{lem-lowerboundentropyperiodic}
    When $(\Omega,f)$ is an expansive homeomorphism, we have
    $$h(f)\geq \limsup_{n\to\infty} \frac{1}{n}\log(p_n(f)).$$
\end{lemma}

\begin{proof}
    Let $\delta$ denote an expansive constant for $f$, and let $\varepsilon<\delta.$ Let $x,y\in P_n(f).$ By expansiveness, for some $i\in\Z$,
    $$d(f^ix,f^iy)>\delta>\varepsilon.$$
    But, since $x,y$ are periodic with period $n$, this means $$d_n(x,y)=\max_{0\leq i < n} d(f^ix,f^iy)>\varepsilon.$$
    That is, $P_n(f)$ is $(n,\varepsilon)$-separated for all $n$, and we get that $$s_n(\varepsilon)\geq p_n(f)$$
    for all $n$. The conclusion follows.
\end{proof}


Recall from the introduction that a finitely presented system is an expansive factor of an SFT. We introduce briefly an equivalent definition of FP given by Fried in \cite{FriedFPSystems}.
Suppose $f:\Omega\to\Omega$ is a factor of an SFT $X\subset \mathcal{A}^\Z$ on alphabet  via $\pi:X\to\Omega$. Then we have an equivalence relation $E$ on $X$ given by
$$xEy \iff \pi (x)=\pi(y),$$
and we can write $\Omega = X/E$. 
Now, if we view $E$ as a subshift of $(\mathcal{A}\times\mathcal{A})^\Z$, \cite{FriedFPSystems} introduces the term finitely presented for such $f$ where $E$ is an SFT.\footnote{The terminology here is loosely based on that of finitely presented groups, in the sense that SFTs are ``like'' finitely generated groups. See \cite{FriedFPSystems} for more details.} The same reference also provides the following fundamental theorem, which in particular gives the equivalence of the two definitions.


\begin{theorem}[Fried \cite{FriedFPSystems}]\label{thm-FPiffMP}
    $\Omega=X/E$ for some SFTs $X,E$ if and only if $(\Omega,f)$ is an expansive factor of an SFT if and only if $\Omega$ admits a Markov partition.
\end{theorem}

These \textit{Markov partitions} have a long history dating back to Sinai in his work on \textit{Anosov diffeomorphisms} \cite{SinaiMarkovPartitions} and work on toral automorphisms by Adler and Weiss \cite{AdlerWeissMarkovPartitions}. They were further developed by Bowen as a tool in better understanding what are known as \textit{Axiom A diffeomorphisms} \cite{BowenEqStatesErgodicTheory}, and generalized by Fried into FP systems. We also point out that \cite{AshleyKitchensBoundariesOM} collects many of the results on Markov partitions in generality.

Often, Markov partitions are constructed with a particular class of dynamical system in mind, so there is some variance in the definitions across the literature. For completeness, we give a definition of Markov partitions following \cite{FriedFPSystems}, who in turn follows \cite{BowenEqStatesErgodicTheory}. The definition is somewhat involved, but the main point for us is that a Markov partition yields an SFT and a factor map with properties outlined in Proposition \ref{prop-MPconsequences}. 

Let $(\Omega,f)$ be an expansive dynamical system, with expansive constant $c>0$. For $\varepsilon>0$, define the $\varepsilon$-stable and $\varepsilon$-unstable sets $$W_\varepsilon^s(x) = \{y\in \Omega: d(f^nx,f^ny)\leq \varepsilon:n\geq 0\},\quad W_\varepsilon^u(x) = \{y\in \Omega: d(f^nx,f^ny)\leq \varepsilon:n\leq 0\}.$$  

Taking $\varepsilon\leq c/2$, $W_\varepsilon^u(x)\cap W_\varepsilon^s(y)$ consists of at most one point for any $x,y\in\Omega$. Let $$D_\varepsilon:\{(x,y)\in \Omega\times \Omega:W_\varepsilon^u(x)\cap W_\varepsilon^s(y)\neq \emptyset\}.$$
A set $R\subset \Omega$ is a \textit{rectangle} if $R\times R\subset D_\varepsilon$. Finally, a \textit{Markov partition} for $(\Omega,f)$ is a finite cover $\mathcal{M}=\{R_i\}_{i=1}^n$ of $\Omega$ consisting of nonempty rectangles with disjoint interiors of diameter less than $\varepsilon$ such that $\overline{\operatorname{int}R}=R$ for all $R\in \mathcal{M}$, satisfying the \textit{Markov property}: Whenever $x\in \operatorname{int}R_i$ and $fx\in \operatorname{int}R_j$, $$f\big(W_\varepsilon^s(x) \cap R_i\big)\subset R_j,\quad f^{-1}\big(W_\varepsilon^s(fx) \cap R_j\big)\subset R_i.$$

Given the Markov partition $\mathcal{M}=\{R_i\}_{i=1}^n$, we define a transition matrix $$A_{ij}=\begin{cases}
    1,& \text{if }\operatorname{int}(R_i)\cap f^{-1}(\operatorname{int}(R_j)) \neq \emptyset, \\
    0,& \text{otherwise,}
\end{cases}$$
and denote the shift space generated by this matrix $X_A$.
We then obtain a factor map $\pi:X_A\to \Omega$ defined by
$$\pi(...x_{-1}x_0x_1...) = \bigcap_{i\in\Z} f^{-i}(R_{x_i}).$$
This map is well-defined, continuous, and makes the following diagram commute. $$\begin{tikzcd}
X \arrow[r, "\sigma_A"] \arrow[d, "\pi"] & X \arrow[d, "\pi"] \\
\Omega \arrow[r, "f"]                        & \Omega                  
\end{tikzcd}$$ In fact, $\pi$ has several other desirable properties.

\begin{proposition}\label{prop-MPconsequences}
    Let $\pi:X\to\Omega$ be a factor map arising from a Markov partition of $\Omega$. Then the following hold.
    \begin{enumerate}[(A)]
    \item\label{prop-MPconsbounded} There exists an $M\in\N$ such that $\sup_{w\in\Omega} |\pi^{-1}w|\leq M$, i.e. $\pi:X\to\Omega$ is \textit{bounded-to-one}. 
    \item\label{prop-MPconsperiodic} If $w\in \Omega$ is periodic, all points in $\pi^{-1}w$ must be periodic. Moreover, any pair of points $x,y\in \pi^{-1}w$ must be mutually separated, meaning $x_i\neq y_i$ for all $i\in\Z$. 
    \item\label{prop-MPconsunique} The set of $w\in\Omega$ such that $|\pi^{-1}w|=1$ is a dense $G_\delta$ set, and in particular nonempty. 
    \item\label{prop-MPconstrans} If $\Omega$ is topologically transitive, then $X$ is irreducible. Moreover, if $\Omega$ is mixing, then $X$ is mixing. 
    
\end{enumerate}
\end{proposition}

Properties (\ref{prop-MPconsbounded}) and (\ref{prop-MPconsperiodic}) are consequences of Lemma 1.8 in \cite{AshleyKitchensBoundariesOM}, sometimes known as the ``no diamonds'' property of Markov partitions. For the latter, see Corollary 11 and Proposition 12 in \cite{ManningRationalZeta}, which refers to \cite{BowenMarkovPartitionsMinimalSets}.

One can also find (\ref{prop-MPconsunique}) in \cite{AshleyKitchensBoundariesOM}. From this property, we will actually only need the existence of a single point with a unique preimage in the presentation in the proof of Theorem \ref{thm-LPSofFP-intro}. Surprisingly, in \cite{MarcusEtAlEmbeddingTheorems}, there is significant interest in arbitrary SFT covers where at least one point has a unique preimage. We will touch on this later in this section.

The forward direction of (\ref{prop-MPconstrans}) is Theorem 3.19 in \cite{BowenEqStatesErgodicTheory} specifically for the Axiom A case, but the same proof generalizes to FP systems. We note also that the converse statements of (\ref{prop-MPconstrans}) are clearly true.

\subsection{Proof of Theorems \ref{thm-LPSofFP-intro} and \ref{thm-LPSofirredSFT-intro}}
We have collected all the necessary tools to prove the next main result. Along the way, we will obtain two intermediate results that are interesting in their own right. Intuitively, we will look at different ``layers'' of the system, where the layers are separated by the number of preimages a point has with respect to the cover arising from a Markov partition. The idea of this proof was first suggested by Doering in \cite{DoeringMasters} as a way of potentially characterizing the LPS of sofic shift spaces.

Let us begin the proof, by letting $(\Omega,f)$ be an FP system. Recall that our goal is to show that
$$\operatorname{LPS}(f)=F\cup\bigcup_{i=1}^n d_iS_i,$$
where $F$ is finite, each $S_i$ is cofinite, and either $F$ or the big union may be empty, but not both.

    Let $\mathcal{M}=\{R_i\}_{i=1}^{|\mathcal{M}|}$ be a Markov partition for $(\Omega,f)$, $X$ the associated SFT with transition matrix $A$, and $\pi:X\to\Omega$ the corresponding factor map. Moreover, let $$\Omega_i:=\{w\in\Omega: \pi^{-1}w \text{ has at least $i$ mutually separated preimages}\},$$
    where it is clear that $\Omega_1=\Omega$. We have that $\Omega_i$ is $f$-invariant, and $$\Omega_i\supset \Omega_{i+1}$$ for all $i$. Note that by Proposition \ref{prop-MPconsequences} (\ref{prop-MPconsbounded}), $\Omega_i$ is nonempty for only finitely many $i$. We will analyze each nonempty $\Omega_i$ separately,
    as if we are ``peeling back layers''.\footnote{Like an onion, perhaps.} The authors of \cite{AshleyKitchensBoundariesOM} chose independently to call $\Omega_2$ the \textit{core} of a Markov partition, invoking a similar image.

    Denote $$\tilde{P}_n(f,\pi):=\{w\in\Omega:f^n(w)=w \text{ and }|\pi^{-1}w|=1\},$$ and let $$\tilde{Q}_n(f,\pi):=\{w\in \tilde{P}_n(f,\pi): \forall k<n,f^kw\neq w \}$$ be the set of all periodic points with unique preimage and \textit{least} period $n$. It is immediate that \begin{equation}\label{Q_n<P_n}
    \tilde{Q}_n(f,\pi)=\tilde{P}_n(f,\pi)\setminus \bigcup_{d|n,d<n} \tilde{P}_d(f,\pi),
    \end{equation}
    which we will use later.

    We will now construct a set of adjacency matrices $A^{(i)}$ such that $X_i:=X_{A^{(i)}}$ is a bounded-to-one factor map onto $\Omega_i$. The idea is inspired by the set of matrices used by Manning to prove that Axiom A diffeomorphisms have rational zeta functions in \cite{ManningRationalZeta}.
    
    We will define $A^{(i)}$ to be the adjacency matrix of a directed graph $G_i$ whose vertices are the set of all $$V\subset\{1,2,3,...,|\mathcal{M}|\}$$  such that $|V|=i$ and $$\bigcap_{l\in V} R_l \neq \emptyset.$$ 
    
    We will draw an edge from vertex $U=\{u_1,...,u_i\}$ to vertex $V=\{v_1,...,v_i\}$ if and only if there exists a reordering of $V$ such that $$A_{u_k,v_{k}}=1$$ for all $1\leq k\leq i$. 

    Recall now that that $\pi(...x_{-1}x_0x_1...)$ is defined to be the unique point $w$ such that $$w\in \bigcap_{n\in\Z} f^{-i}(R_{x_n}).$$ There is a natural bounded-to-one factor map $\pi_i:X_i\to\Omega_i$ induced from $\pi$. If $x\in X_i$, then $x$ corresponds to at least $i$ mutually separated points $y^1,y^2,...,y^i\in X_1$, such that for all $n\in\Z$, $$\bigcap_{1\leq k\leq i} R_{y^k_n} \neq \emptyset.$$ Lemma 1 from \cite{ManningRationalZeta} deduces directly from the definition of $\pi$ and expansiveness of $f$ that $$\pi y^1=\pi y^2=\cdots=\pi y^i=w\in \Omega,$$ and we may thus set $$\pi_i x:=w,$$ to obtain the following commutative diagram. 
    $$\begin{tikzcd}
X_i \arrow[r, "\sigma_{A^{(i)}}"] \arrow[d, "\pi_i"] & X_i \arrow[d, "\pi_i"] \\
\Omega_i \arrow[r, "f|_{\Omega_i}"]                        & \Omega_i                  
\end{tikzcd}$$
    Since $X_i$ is compact, and $\pi_i$ is continuous, $\Omega_i$ is a compact subset of $\Omega$. 
    \begin{remark}
    In general, it is possible that for an edge in $G_i$ from $U$ to $V$, there exists some $u\in U$ such that $A_{uv}=1$ for two distinct $v\in V$, or some $v\in V$ such that $A_{uv}=1$ for two distinct $u\in U$. 
    
    However, if this happens, the edge from $U$ to $V$ cannot be part of an irreducible component of $G_i$. If it were, the edge would be part of a cycle, which contradicts Proposition \ref{prop-MPconsequences} (\ref{prop-MPconsperiodic}). If $x\in X_i$ is contained in an irreducible component, it thus corresponds to exactly $i$ mutually separated points in $X_1$. We will later only consider irreducible components of $X_i$, so we need not be concerned about the scenario above.
    \end{remark}

    Whenever we write $f$ in the remainder of the proof, we will always mean $f$ restricted to the appropriate invariant subset of $\Omega$. When there is no ambiguity, we allow ourselves to refrain from explicitly stating this restriction to avoid cumbersome notation.

    We briefly put the proof on hold to remark that we have effectively just shown that $(\Omega_i,f|_{\Omega_i})$ is FP for all $i$. Indeed, since $\Omega_i$ is a factor of the SFT $X_i$, the only other requirement is that $f|_{\Omega_i}$ is expansive by Theorem \ref{thm-FPiffMP}, but this is immediate since $f$ is expansive. This fact was already shown by a similar argument for $(\Omega_2,f|_{\Omega_2})$ in Observation 3.1 of \cite{AshleyKitchensBoundariesOM}. We encode this intermediate result as a theorem. 
    \begin{theorem}
    Let $(\Omega,f)$ be an FP system with SFT cover $X$ and factor map $\pi:X\to\Omega$ arising from a Markov partition. Let $$\Omega_i:=\{w\in\Omega: \pi^{-1}w \text{ has at least $i$ mutually separated preimages}\}.$$ Then if $\Omega_i$ is nonempty, $(\Omega_i,f|_{\Omega_i})$ is an FP system.
\end{theorem}

    We return to the proof. Since $\pi$ is bounded-to-one by Proposition \ref{prop-MPconsequences} (\ref{prop-MPconsbounded}), so is $\pi_i$ for all $i$, although the bounds vary in $i$. Indeed, if $$M(i):=\max\{|\pi_i^{-1}w|:w\in\Omega_i\},$$ then $M(i)$ may be as high as $\binom{M(1)}{i}$. 
    
    If $w\in\Omega$ is periodic and has exactly $i$ preimages under $\pi$, then these preimages are mutually separated by Proposition \ref{prop-MPconsequences} (\ref{prop-MPconsperiodic}), and $w$ must have a unique preimage under $\pi_i$.  In other words, if we write $M:=M(1)$, then 
    $$\operatorname{LPS}(f)=\bigcup_{1\leq i \leq M} \operatorname{supp}\left(\left|\tilde{Q}_n\left(f,\pi_i\right)\right|\right).$$ 

    Note also that if $j<i$, then $w$ must have at least two mutually separated preimages with respect to $\pi_j$, and  if $j>i$, then $w\notin \Omega_j$.

    For any $i$, it is possible that $X_i$ is not irreducible. However, since $X_i$ is an SFT, it consists of finitely many irreducible components and possibly some transient points, cf. Section 4.4 in \cite{Lind_Marcus_1995}.  Call these irreducible components $X_i^j$ for $1\leq j \leq k_i,$ where $k_i<\infty$. We will consider the images of these irreducible components separately.

    If a periodic point $w\in\Omega$ has a unique preimage $x$ under $\pi_i$, $x$ must be periodic, and hence cannot be transient. That is, $x$ must be in $X_i^j$ for some unique $j$. Thus, if we write $$\pi_i^j:=\pi_i|_{X_i^j},$$ we get $$\operatorname{LPS}(f)=\bigcup_{1\leq i \leq M}\bigcup_{1\leq j \leq k_i} \supp\left(\left|\tilde{Q}_n\left(f,\pi_i^j\right)\right|\right).$$

    Now, it is certainly true that a periodic point $w\in\Omega$ having $i$ preimages under $\pi$ means it has a unique preimage under $\pi_i$, and therefore that preimage $x$ must lie in $X_i^j$ for a unique $j$. However, it may transpire that $w$ has more than $i$ preimages under $\pi$, but a unique preimage with respect to $\pi_i^j$ for some (or several) $j$. For example, two preimages of $w$ with respect to $\pi_i$ could lie in separate irreducible components of $X_i$. In other words, $$\left\{\tilde{Q}_n\left(f,\pi_i^j\right)\right\}_{1\leq i \leq M}^{1\leq j \leq k_i}$$ need not form a partition of $Q_n(f)$. 
    
    Suppose a periodic point $w\in\Omega_i^j:=\pi_i^j(X_i^j)$ does not have a unique preimage with respect to $\pi_i$ (note the lack of superscript). Then $w$ will also be in $\Omega_{i'}^{j'}$ for some $i' > i$ and $j'\leq k_{i'}$. We can thus ignore the period of this point when analyzing $\Omega_i$, and handle it when we get to $\Omega_{i'}$. This means that if $\Omega_i^j$ contains no points with a unique preimage under $\pi_i$, we can ignore the index pair $(i,j)$, without losing any periods.
    
    At the same time, if there is $w\in \Omega_i^j$ such that $|\pi_i^{-1}w|=1$, then $\pi_i^{-1}w\in X_i^j$. Thus, if we denote $$\mathcal{G}:=\left\{(i,j):1\leq i\leq M,\;1\leq j\leq k_i,\; \exists w\in\Omega_i^j:|\pi_i^{-1}w|=1\right\},$$
    we can write
    $$\operatorname{LPS}(f)=\bigcup_{(i,j)\in\mathcal{G}}\supp\left(\left|\tilde{Q}_n\left(f,\pi_i^j\right)\right|\right).$$

    To ease the notation, we will write $$\tilde{p}_{n,i,j}:=\left|\tilde{P}_n\left(f,\pi_i^j\right)\right|, \quad \text{ and } \quad\tilde{q}_{n,i,j} := \left|\tilde{Q}_n\left(f,\pi_i^j\right)\right|,$$  in what follows.

    At this point, if we can show that for all $(i,j)\in\mathcal{G}$, 
    $$\supp\left(\tilde{q}_{n,i,j}\right)$$ is either a singleton $\{d\}$ or $dS$ for a cofinite set $S$ and a positive integer $d$, we are effectively done with the proof. This is exactly what we will do.

    Consider two cases. 
    If $X_i^j$ is a finite orbit of length $d$, then $\Omega_i^j$ is also a finite orbit of length $d$, because $(i,j)\in\mathcal{G}$. Hence,   
    $$\supp\left(\tilde{q}_{n,i,j}\right)=\{d\}.$$
    
    Otherwise, $X_i^j$ does not consist of a single orbit and some point in $\Omega_i^j$ has a unique preimage under $\pi_i$. We will see that this means $$\supp\left(\tilde{q}_{n,i,j}\right)=dS,$$
    where $S$ is cofinite and $d$ is the global period of $X_i^j$.

    To this end, we will first show that$$\lim_{n\to\infty} \frac{1}{dn} \log\left(\tilde{p}_{dn,i,j}\right)=\log(\lambda),$$ 
    which we will abbreviate 
    $$\tilde{p}_{dn,i,j}\sim \lambda^{dn},$$
    where $\lambda$ is such that  $h(X_i^j)=\log(\lambda)$. 

    Let $$\Omega':=\Omega_{i+1}\cap \Omega_i^j$$ and $X':=(\pi_i^j)^{-1}\left(\Omega'\right)$. Since $\Omega_{i+1}$ is closed and invariant, so is $\Omega'$. Moreover, $\Omega'$ is a proper subsystem of $\Omega_i^j$, because $(i,j)\in\mathcal{G}$, so $X'$ is a proper subshift of $X_i^j$. Recall also that if $w\in\Omega_i^j$ is periodic but has strictly more than $i$ mutually separated preimages with respect to $\pi$, then $w\in \Omega'$.

     Let $\mu$ denote the real number such that $h(X')=\log(\mu)$. We first observe then that  $f|_{\Omega'}$ is expansive because $f$ is, and since $\pi_i$ is bounded-to-one we get that 
     \begin{align*}
         \limsup_{n\to\infty}\frac{1}{n}\log|P_n(f|_{\Omega'})| \stackrel{\ref{lem-lowerboundentropyperiodic}}{\leq} h(f|_{\Omega'}) \stackrel{{\ref{lem-bdtoonepreservesentropy}}}{=}h(X') .
     \end{align*}
    Moreover, since $X'$ is a proper subsystem of $X_i^j$, by Lemma \ref{lem-SFTentropyminimal}, 
    \begin{equation}\label{mu<lam}
        \mu < \lambda.
    \end{equation}
    Now, for all $\varepsilon>0$, we have that for all $n$ large enough,
    $$\frac{1}{n}\log|P_n(f|_{\Omega'})| < \log(\mu)+\varepsilon,$$
    or equivalently
    \begin{equation*}
        |P_{n}(f|_{\Omega'})| < (\mu e^\varepsilon)^n.
    \end{equation*}
    We will pick $\varepsilon$ such that
    $$\mu e^\varepsilon < \lambda,$$
    which is possible by Equation (\ref{mu<lam}).     
     
     As a consequence of Theorem 4.5.11 in \cite{Lind_Marcus_1995} there exist $C, D>0$ such that for all large enough $n\in \N$,
     $$C\cdot \lambda^{dn} \leq |P_{dn}(X_i^j)| \leq D \cdot \lambda^{dn}.$$
     If $w\in\Omega_i^j$ has period $n$ and a unique preimage in $X_i^j$, this preimage must also be periodic with period $n$. Since $d$ is the global period of $X_i^j$, this means that a periodic point $w$ can only have a unique preimage if its period is a multiple of $d$. This yields
     \begin{align*}
         \tilde{p}_{dn,i,j} \leq |P_{dn}(X_i^j)| \leq D\cdot \lambda^{dn},
     \end{align*}
     for large enough $n$.
     
     On the other hand, every periodic point $w\in\Omega_i^j$ that does not have a unique preimage under $\pi_i^j$ lies in $\Omega'$. We remarked earlier that the converse need not be true, so while $$P_{dn}\left(f|_{\Omega_i^j}\right) = P_{dn}(f|_{\Omega'}) \cup \tilde{P}_{dn}\left(f,\pi_i^j\right),$$
     the sets on the right hand side need not be disjoint. We get the inequality $$\tilde{p}_{dn,i,j}+|P_{dn}(f|_{\Omega'})|\geq\left|P_{dn}\left(f|_{\Omega_i^j}\right)\right|.$$ 
     
     Next, the image of any periodic point of period $dn$ in $X_i^j$ must also have (not necessarily least) period $dn$. Some such points in $X_i^j$ may be collapsed into a single point in $\Omega_i^j$, but we know that $\pi_i$ is bounded-to-one with bound $M(i)$. We get $$\frac{1}{M(i)}|P_{dn}(X_i^j)|\leq \left|P_{dn}\left(f|_{\Omega_i^j}\right)\right|.$$
     Therefore, for all sufficiently large $n$,
     \begin{align*}
         \tilde{p}_{dn,i,j} &\geq \left|P_{dn}\left(f|_{\Omega_i^j}\right)\right| - |P_{dn}(f|_{\Omega'})| \\
         &\geq \frac{1}{M(i)}|P_{dn}(X_i^j)|-|P_{dn}(f|_{\Omega'})| \\
         &\geq \frac{C}{M(i)}\cdot\lambda^{dn} - (\mu e^\varepsilon)^{dn} \\
         &= \lambda^{dn}\left(\frac{C}{M(i)}-\left(\frac{\mu e^\varepsilon}{\lambda}\right)^{dn}\right)
     \end{align*}
     By our choice of $\varepsilon$,
     $$\left(\frac{\mu e^\varepsilon}{\lambda}\right)^{dn}\to0$$
     as $n\to\infty$, and we conclude that indeed
     $$\tilde{p}_{dn,i,j}\sim\lambda^{dn}.$$

    At this stage, we pause the proof once again to collect an intermediate result.

    \begin{theorem}\label{thm-bdtooneentropyislimofperiodic}
    Let $(\Omega,f)$ be a topologically transitive FP system, and let an irreducible SFT $X$ and factor map $\pi:X\to\Omega$ be those arising from a Markov partition. Let $$\tilde{P}_n(f,\pi):=\{w\in\Omega:f^n(w)=w \text{ and }|\pi^{-1}w|=1\},$$ and let $d$ be the global period of $X$. Then the following quantities are equal. \begin{enumerate}[(1)]
        \item $h(f)$.
        \item $\lim_{n\to\infty}\frac{1}{dn}\log(|\tilde{P}_{dn}(f,\pi)|).$
        \item $\lim_{n\to\infty}\frac{1}{dn} \log (p_{dn}(f))$.
        \item $\limsup_{n\to\infty} \frac{1}{n}\log(p_n(f))$.
    \end{enumerate} 
\end{theorem}
\begin{proof} By Proposition \ref{prop-MPconsequences} (\ref{prop-MPconsunique}), there is at least one point in $\Omega$ with a unique preimage under $\pi$. Hence, $\Omega_2$ is a strict subsystem of $\Omega$. Since $X$ is irreducible by Proposition \ref{prop-MPconsequences} (\ref{prop-MPconstrans}), following the above argument gives that \textit{(1)}=\textit{(2)}, which we use to obtain the remaining equalities. Note that $$P_n(f)\supset\tilde{P}_{n}(f,\pi),$$
and hence $$p_n(f)=|P_n(f)|\geq |\tilde{P}_n(f,\pi)|.$$
A direct chain of inequalities completes the proof:
    \begin{align*}
        h(f)&\stackrel{\ref{lem-lowerboundentropyperiodic}}{\geq} \limsup_{n\to\infty} \frac{1}{n} \log(p_n(f)) \\
        &\geq \limsup_{n\to\infty} \frac{1}{dn}\log(p_{dn}(f)) \\
        &\geq \liminf_{n\to\infty} \frac{1}{dn}\log(p_{dn}(f)) \\
        &\geq \liminf_{n\to\infty} \frac{1}{dn}\log(|\tilde{P}_{dn}(f,\pi)|) \\
        &=\limn\frac{1}{dn}\log(|\tilde{P}_{dn}(f,\pi)|)=h(f) 
    \end{align*}
\end{proof}

\begin{remark}
    In the statement above, the value of $d$ does not depend on the choice of Markov partition, but is intrinsic to the FP system. Proposition 5.1 of \cite{MarcusEtAlEmbeddingTheorems} shows in particular that the global periods of all irreducible SFT covers of a sofic shift for which there exist a point with a unique preimage must be equal. While their result is stated for sofic shifts, the proof carries over to the FP case directly. Combining this with Proposition \ref{prop-MPconsequences} (\ref{prop-MPconsunique}) shows that the value of $d$ is independent of the choice of Markov partition.
\end{remark}

    Let us continue with the proof of Theorem \ref{thm-LPSofFP-intro}. Recall that $$\tilde{p}_{n,i,j}:=\left|\tilde{P}_n\left(f,\pi_i^j\right)\right| \quad \text{ and } \quad\tilde{q}_{n,i,j} := \left|\tilde{Q}_n\left(f,\pi_i^j\right)\right|,$$ and our aim is to show that $$\supp\left(\tilde{q}_{n,i,j}\right)=dS,$$
    where $S$ is cofinite and $d$ is the global period of $X_i^j$. 
    
    We have shown that
    $$\lim_{n\to\infty} \frac{1}{dn} \log\left(\tilde{p}_{dn,i,j}\right)=\log(\lambda),$$
    where $\lambda$ is such that  $h(X_i^j)=\log(\lambda)$ and $d$ is the global period of $X_i^j$. We apply this fact in conjunction with a technique from the proof of Corollary 4.3.8 in \cite{Lind_Marcus_1995}. 
    
    First observe that from Equation (\ref{Q_n<P_n}), we get that \begin{equation}\label{analyticarg}
         \tilde{q}_{n,i,j} \geq \tilde{p}_{n,i,j}-\sum_{k|n,k<n}\tilde{p}_{k,i,j} \geq \tilde{p}_{n,i,j}-\sum_{k<n/2}\tilde{p}_{k,i,j}.
     \end{equation}
     
     Moreover, we note that $\tilde{P}_n(f,\pi_i^j)$ is empty whenever $n$ is not a multiple of $d$, since a periodic point $w$ with a unique preimage has period $n$ if and only if its preimage has period $n$. Hence $\tilde{Q}_n(f,\pi_i^j)$ is empty if $n\notin d\N$. 
     
     Since also $\tilde{p}_{dn,i,j}\sim \lambda^{dn}$, we conclude from equation (\ref{analyticarg}) that $\tilde{Q}_{dn}(f,\pi_i^j)$ is nonempty for all large enough $n$. That is, 
     $$\supp\left(\tilde{q}_{n,i,j}\right)=dS,$$
     where $S$ is cofinite. 

    We have now shown that
    \begin{align*}
        \operatorname{LPS}(f)&=\bigcup_{(i,j)\in\mathcal{G}} \supp\left(\left|\tilde{Q}_n\left(f,\pi_i^j\right)\right|\right) \\
        &= F\cup \bigcup_{l=1}^L d_lS_l
    \end{align*}
    where $F$ is finite and each $S_l$ is cofinite. Since $X=X_1$ is an SFT, it has at least one periodic point, and therefore $\Omega$ has at least one periodic point. Hence, either $F$ or the big union may be empty, but not both. This concludes the general case.
         
Recall from Proposition \ref{prop-MPconsequences} (\ref{prop-MPconstrans}) that if $(\Omega,f)$ is FP and topologically transitive, $X$ can be taken to be irreducible. Such an SFT $X$ must either consist of a single finite orbit, or have infinitely many periodic points. Consequently, the LPS of a topologically transitive FP system must be of the form above, but if it is finite, it must be a singleton.

Moreover, if $(\Omega,f)$ is topologically mixing, Proposition \ref{prop-MPconsequences} (\ref{prop-MPconstrans}) says that $X_1$ is mixing, hence $X_1$ has period 1 by Proposition 4.5.10 in \cite{Lind_Marcus_1995}. In the notation above, this means that $d_l=1$ for some $l$. The union of a cofinite set with any other subset of the natural numbers is still cofinite, so the contribution of $\tilde{q}_{n,i,j}$ for $i>1$, the ``deeper'' layers of the system, to the LPS is effectively hidden. With that, we have proven Theorem \ref{thm-LPSofFP-intro}.

Suppose now that $\Omega$ is an SFT. Then it is straightforward to show that the cylinder sets formed by the alphabet of $\Omega$ are a Markov partition, where $\pi$ is the identity map. That is, $X=\Omega$. We thus only need to consider the irreducible components $X_1^1,...,X_1^{k_1}$, noting again that any periodic point is not transient, so it must lie in one of these irreducible components. By the above analysis, the LPS of each component is either a singleton $\{d\}$ or $dS$ for $S$ cofinite. Thus,
$$LPS(X)=F\cup \bigcup_{l=1}^Ld_lS_l,$$
where $F$ is finite and each $S_l$ cofinite. Again, either $F$ or the big union may be empty, but not both. This classification was first proved in \cite{DoeringPavlov} by a direct combinatorial argument. 

The observations we made with respect to the irreducibility hold in the same way. If $X$ is irreducible, then we only have one component $X_1^1$, and $LPS(X)$ is either a singleton $\{d\}$ or $dS$ for $S$ cofinite. If $X$ is mixing, then $d=1$, and $LPS(X)$ is either $\{1\}$ or a cofinite set $S$. This concludes the proof of Theorem \ref{thm-LPSofirredSFT-intro}.

\subsubsection{Receptive points}

In the introduction, we mentioned that our main results are motivated by Krieger's famed embedding theorem. It is natural to ask for generalizations of this theorem. In \cite{BoyleLowerEntropy}, Boyle proves one direction of Krieger's theorem in the case of mixing sofic shifts, and recent work by Marcus, Meyerovitch, Thomsen, and Wu in \cite{MarcusEtAlEmbeddingTheorems} explores the irreducible case for both sofic shifts and SFTs. In both cases, the main ingredient is the set of \textit{receptive periodic points}. 

Let $Y$ be a sofic shift. A word $v$ is \textit{synchronizing} if for any words $w$, $w'$, such that $wv$ and $vw'$ are both allowed in the shift, then $wvw'$ is also allowed. A periodic point $w^\infty$ in $Y$ with least period $|w|$ is then said to be \textit{receptive} if there exist synchronizing words $v_1,v_2$ such that the word $v_1w^kv_2$ is allowed in $X$ for all $k\in\N$.

Briefly, the intuition behind the utility of receptive periodic points is as follows. Recall that in order to embed a shift space $X$ into a mixing SFT $Y$, Krieger's embedding theorem requires that $h(X)<h(Y)$ and $q_n(X)\leq q_n(Y)$ for all $n\in\N$. Both conditions force our embedding to behave a certain way on the set of all periodic points: The former gives rise to an embedding into $P_n(Y)$ for all but finitely many values of $n$, while the latter condition forces an embedding on these remaining values of $n$. A priori, these two embeddings need not cooperate nicely if $X$ is a mixing sofic shift. Receptive periodic points are the key element that allow us to ``glue together'' the two embeddings that arise from the two conditions, which is precisely what is done in Theorem 3.6 of \cite{BoyleLowerEntropy}. 

For our purposes, it will be convenient to use the following definition of receptive periodic points, which is equivalent to the definition above by Theorem 4.7 in \cite{MarcusEtAlEmbeddingTheorems}. 

\begin{definition}\label{def-receptive}
    Let $Y$ be a sofic shift and $w\in Y$ be periodic. We say $w$ is receptive if there exists an irreducible SFT $X$ and a factor map $\pi:X\to Y$ such that $w$ has a preimage under $\pi$ with the same least period. In particular, periodic points with a unique preimage under some such factor map $\pi$ are receptive.
\end{definition}

We also extract from Proposition 5.1 of the same reference the equivalence of two ways in which one might define the ``period'' of a sofic shift. Recall that the global period of an irreducible SFT $X$ is defined to be $\gcd(\operatorname{PS}(X)),$ which is equal to the period of any graph presentation of $X$.

\begin{proposition}[\cite{MarcusEtAlEmbeddingTheorems}]\label{prop-periodequivalence}
    Let $Y$ be an irreducible sofic shift. Then the global period of any irreducible SFT $X$ for which there exists a factor map $\pi:X\to Y$ which is one-to-one on at least one point in $Y$ is equal to the gcd of the periods of all receptive periodic points.
\end{proposition}

This leads us to the following. Suppose $Y$ is an irreducible sofic shift that is not a finite orbit. $Y$ is a topologically transitive FP system, and hence admits a Markov partition. Let $X$ be the SFT cover arising from this Markov partition for $Y$ and denote the corresponding factor map $\pi:X\to Y$. Denote also by $\tilde{Q}_n(Y)$ the set of points of least period $n$ which have a unique preimage under $\pi$. 

By Proposition \ref{prop-MPconsequences} (\ref{prop-MPconsunique}) there is at least one such point, so we know from the proof of Theorem \ref{thm-LPSofFP-intro} that 
$$\supp(\{|\tilde{Q}_n(Y)|\})=dS',$$
where $S'$ is some cofinite set, and $d$ is the global period of $X$. By Definition \ref{def-receptive}, any point with a unique preimage in this particular cover of $Y$ is receptive, and thus
$$\operatorname{LPS}(\{y\in Y:y \text{ is periodic and receptive}\})\supset dS'.$$
However, by Proposition \ref{prop-periodequivalence}, we also know that
$$\operatorname{LPS}(\{y\in Y:y \text{ is periodic and receptive}\})\subset d\N.$$

Of course, in the case that $Y$ is a finite orbit, all of its points are receptive and of the same period. Combining all of the above, we immediately get the following result.

\begin{theorem}\label{thm-LPSofreceptive}
    Let $Y$ be an irreducible sofic shift. Then $$\operatorname{LPS}(\{y\in Y:y \text{ is periodic and receptive}\})$$ is either a singleton $\{d\}$ or $dS$, 
    where $S$ is a cofinite set.
\end{theorem}

We point out that Definition \ref{def-receptive} is not symbolic, and hence a priori can be used to define receptive points in an FP system. In fact, although \cite{MarcusEtAlEmbeddingTheorems} focuses on sofic systems, it seems many of their results carry over directly to the FP case. We thus conjecture that Theorem \ref{thm-LPSofreceptive} also holds for FP systems.

\subsection{Gap shifts}\label{sec-LPSofgap}

We shall obtain a more precise characterization of the LPS of a subclass of shift spaces called \textit{gap shifts}. We denote $\N_0:=\N\cup\{0\}$. 

\begin{definition}
    Let $S\subset \mathbb{N}_{0}$. Then the $S$-gap shift, denoted $X(S)$, is the shift on $\{ 0,1 \}$ such that the word $10^m1$ is only allowed if $m\in S$.
\end{definition}

Note that if $S$ is finite, the symbol 1 occurs infinitely many times to the right and to the left in any point $x\in X(S)$. Moreover, if $S$ is infinite, infinitely many consecutive 0's are allowed, and hence the point $0^\infty\in X(S)$. 

The intersection of $S$-gap shifts and sofic shifts (respectively SFTs) is easily described, due to the following result due to Dastjerdi and Jangjoo. We say a set $S\subset \N_0$ is \textit{eventually periodic} if, up to some finite number of elements, it consists only of tails of infinite arithmetic progressions. 

\begin{proposition}[\cite{DASTJERDI20122654}]\label{prop-classifySgaps}
    An $S$-gap shift is sofic if and only if $S$ is eventually periodic. Moreover, an $S$-gap shift is an SFT if and only if $S$ is finite or cofinite.
\end{proposition}

A subset $R$ of the natural numbers is \textit{almost sum-closed} if for any pair of \textit{distinct} elements of $R$, their sum is also in $R$. Given a set $S$, we denote its \textit{almost sum-closure} by 
$$
\Sigma_{S}^*:=\left\{  \sum_{i=1}^kn_{i}:k\in\N,n_{i}\in S \text{ with } k=1 \text{ or } \exists i,j\in\{1,...,k\} \text{ s.t. }n_{i}\neq n_{j}  \right\}.
$$
The following properties of almost sum-closures are immediate from the definition and Bézout's Lemma, Theorem 2.10 in \cite{judsonAbstractAlgebraTheory2024}.
\begin{proposition}\label{prop-almostsumprops} 
    If $S\subset \N$ is a singleton, then $$\Sigma_S^*=S.$$ 
    More generally, if $S\subset\N$ is almost sum-closed, then
    $$\Sigma_{S}^*=S.$$ 
    Finally, if $S$ contains at least two distinct elements, then  $$
\Sigma_{S}^*=\gcd(S)\cdot S',
$$
where $S'$ is cofinite. 
\end{proposition}

Recall that a word $w$ is \textit{not purely periodic} if there is no other word $v$ and integer $k>1$ such that $w=v^k$. In other words, if $w$ is not purely periodic and $w^\infty$ is allowed in a shift, then the least period of $w^\infty$ is $|w|$. Our classification of the LPS of gap shifts is then as follows.

\begin{theorem}\label{thm-LPSofgap}
    Let $S\subset\N_0$ be nonempty and let $X(S)$ be the $S$-gap shift. Then $$\operatorname{LPS}(X(S))=\Sigma_{S+1}^*$$
    if $S$ is finite, and $$\operatorname{LPS}(X(S))=\Sigma_{S+1}^*\cup \{1\}$$ if $S$ is infinite. In particular, $\operatorname{LPS}(X(S))$ is either \begin{itemize}
        \item A singleton $\{ d \}$, or
        \item An almost sum-closed set $dR$ with $R$ cofinite, or
        \item An almost sum-closed set $dR$ with $R$ cofinite, together with the singleton $\{ 1 \}$.
    \end{itemize}
\end{theorem}

\begin{proof}
    We will consider three cases.
    First, if $S=\{d\}$, then
    $$X(S)=\{\sigma^i((10^d)^\infty):i\in\Z\},$$
    which is a finite orbit of length $d+1$. Hence,
    $$\operatorname{LPS}(X(S)) = \{d+1\}=\Sigma_{\{d+1\}}^*.$$

    Suppose now that $S$ is finite with $|S|>1$, and let $x\in X(S)$ be periodic. Since $S$ is finite, $x\neq 0^\infty$. Then we either have
    $$x=\sigma^i((10^n)^\infty)$$ for some $i\in\Z$ $n\in S$, or
    $$x=\sigma^i((10^{n_{1}}10^{n_{2}}1\cdots 10^{n_{k}})^\infty),$$
    where $i\in\Z$, $k>1$, $n_j\in S$ for all $1\leq j\leq k$, and $w:=10^{n_{1}}10^{n_{2}}1\cdots 10^{n_{k}}$ is not purely periodic. 
    In the first case, the least period of $x$ is $$|10^n|=n+1\in \Sigma^*_{S+1}.$$
    In the second case, at least two $n_j,n_{j'}$ are distinct, so the least period of $x$ is $$|10^{n_1}1\cdots10^{n_k}|=\sum_{j=1}^k(n_j+1)\in \Sigma^*_{S+1}.$$
    Thus, $\operatorname{LPS}(X(S))\subset \Sigma^*_{S+1}.$
    Conversely, let $\tilde{n}\in \Sigma_{S+1}^*$. Then there exists a set $\{n_j\}_{j=1}^k$ such that $n_j\in S$ for all $1\leq j\leq k$ and at least two $n_j$,$n_{j'}$ are distinct, and $$\sum_{j=1}^k(n_j+1)=\tilde{n}.$$ Then, after sorting the $n_j$ in weakly increasing order, the word $$10^{n_1}1\cdots10^{n_k}$$ is not purely periodic, so
    $$x=(10^{n_1}1\cdots10^{n_k})^\infty$$ is in $X(S)$ and has least period $\tilde{n}$. Using Proposition \ref{prop-almostsumprops}, we conclude that $$\operatorname{LPS}(X(S))=\Sigma_{S+1}^*=\gcd(S+1)\cdot R$$ for some cofinite set $R$.

    Finally, suppose $S$ is infinite. We may then repeat exactly the argument above, but must also note that $0^\infty\in X(S)$. This yields a point of period 1. In this case, we thus get $$\operatorname{LPS}(X(S))=\Sigma_{S+1}^*\cup\{1\}=\left(\gcd(S+1)\cdot R\right)\cup\{1\},$$ where again $R$ is cofinite.

\end{proof}

In the next section, we will show that any set of the form in Theorem \ref{thm-LPSofgap} can in fact be realized by a sofic gap shift, or even a gap shift of finite type.

\section{Realization}\label{realizations}

This section is dedicated to explicit constructions that prove the sufficiency in our main results. In fact, each construction will be a shift space. Before we begin with the LPS constructions, we point out that any construction in this section can be used to realize a desired PS as well. 

For a set $S\subset \N$, we adopt terminology of \cite{DoeringPavlov} and say $S$ is \textit{closed under $\N$-multiples} if for any $s\in S$ and $k \in \N$, $k\cdot s\in S$. It is immediate from the definition that for a given dynamical system $(X,f)$, $$\operatorname{PS}(f)=\{n\cdot s:n\in \N, s\in \operatorname{LPS}(f)\},$$ so $\operatorname{PS}(f)$ is closed under $\N$-multiples. The theorem statements in this section are of the form
\begin{itemize}
    \item[] Given $S\subset \N$ with property $\mathcal{P}$, there is a shift space $X$ of type $\mathcal{T}$ such that $\operatorname{LPS}(X)=S$.
\end{itemize} The described properties $\mathcal{P}$ are compatible with being closed under $\N$-multiples, so we can replace the above with \begin{itemize}
    \item[] Given a $S\subset \N$ with property $\mathcal{P}$ that is closed under $\N$-multiples, there is a shift space $X$ fo type $\mathcal{T}$ such that $\operatorname{PS}(X)=S$.
\end{itemize} 

The constructions are identical, and conclude by noting that if $S$ is closed under $\N$-multiples and $\operatorname{LPS}(X)=S$, then $\operatorname{PS}(X)=S$. Of course, one can ignore finite sets $S$: a finite set that is closed under $\N$-multiples must be empty. 

In particular, this observation means that Theorem \ref{thm-realizeLPSsofic} allows us to construct a sofic shift that realizes a period set of the form outlined in Theorem \ref{thm-PSofrationalzeta}.

\subsection{SFTs}

The following result is a converse to Theorem \ref{thm-LPSofirredSFT-intro}, which is also due to Doering and Pavlov. For completeness, we repeat their constructive argument here. For a graph $G$, a \textit{simple cycle} is a cycle where every vertex has exactly one incoming edge, and exactly one outgoing edge. We call the \textit{edge shift} of $G$ the collection of all bi-infinite walks on $G$ encoded as sequences of edges. We denote this shift $X_G$.

\begin{theorem}[\cite{DoeringPavlov}]\label{thm-realizeLPSirredSFT}
    Suppose $R$ is a singleton $\{  d \}$ or $dS$ for $S$ cofinite. Then there exists an irreducible SFT whose LPS is $R$. Moreover, if $R$ is $\{1\}$ or a cofinite set, there exists a mixing SFT whose LPS is $R$.
\end{theorem}

\begin{proof}
    We will construct a graph $G$ such that $X_G$ is irreducible and $$\operatorname{LPS}(X_G)=R.$$

    If $R=\{ d \}$, let $G$ be a simple cycle of length $d$. Then the edge shift of $G$ consists of $d$ points each with least period $d$, so $\operatorname{LPS}(X_G)=\{d\}$.

Suppose now that $R=dS$ with $S$ cofinite. We can write $$S=F\cup \{ N,N+1,N+2,\dots \}$$ where $F=\{f_1,...,f_k\}$ is finite, $f_j<f_{j+1}<N$ for all $1\leq j < k$. We begin by constructing a graph $G$ such that $X_G$ is a \textit{mixing} SFT with $\operatorname{LPS}(X_G)=S$. First, construct a simple cycle of length $N$. For each $f_j\in F$, attach a simple cycle of length $f_j$ to this first cycle, so that each vertex in the $N$-cycle has at most one smaller cycle attached to it. Each of these smaller cycles gives a periodic point with least period $f_j$, so $$F\subset\operatorname{LPS}(X_G).$$

Recall that $f_1=\min F$. For all $i\in \{1,2,...,f_1-1\},$ attach a simple cycle of length $N+i$ to the $i$-th vertex in the $f_1$-cycle, completing the graph construction. See Figure \ref{fig:sft} for a visualization, where each cycle is labeled by its length. 

If $n\geq N$, we can write $n=qf_1+N+i$ for some $q\in \N_0, 0\leq i <f_1$. By traversing the $N+i$-cycle once, and the $f_1$-cycle $q$ times, we obtain a point with least period $n$, hence
$$\{N,N+1,N+2,...\}\subset \operatorname{LPS}(X_G).$$

It remains to show that the least period of any periodic point in $X_G$ is in $S$. Suppose that $x\in X_G$ is periodic, and consider two cases. If $x$ only traverses a simple cycle, then this must either be an $f$-cycle for $f\in F$ or an $N+i$-cycle for $0\leq i < f_1$. In both cases, the least period of $x$ is in $S$ by construction. 

Suppose now that $x$ traverses more than one simple cycle from our construction. Then, necessarily the least period of $x$ must be at least $N$. However, any such natural number is in $S$ by assumption, so we have exactly realized the desired LPS. Since $G$ has two coprime cycles, one of length $N$ and one of length $N+1$, the graph period of $G$ is 1, and hence $X_G$ is mixing.

\begin{figure}
\begin{center}
\begin{tikzpicture}
    \foreach \x in {0,...,2} {
            \node[style={circle,thick,draw}] (\x) at ({-0.5+3*cos(\x/3 * 360)},{3 * sin(\x/3 * 360}) {};
        }
    \path [->, thick, bend right=50, dashed] (0) edge node[above] {} (1);
    \path [->, thick, bend right=50, dashed] (1) edge node[above] {} (2);
    \path [->, thick, bend right=50, dashed] (2) edge node[above] {} (0);
    \node[style={circle}] (o) at (-0.5,0) {$N$};

    \foreach \x in {-1,...,1} {
            \node[style={circle,thick,draw}] (\x) at ({4+1.5*cos(\x/4 * 360+180)},{1.5 * sin(\x/4 * 360+180}) {};
        }
     \node[style={circle}] (-2) at ({4+1.5*cos(-3/8 * 360+180)},{1.5 * sin(-3/8 * 360+180}) {};
     \node[style={circle}] (2) at ({4+1.5*cos(3/8 * 360+180)},{1.5 * sin(3/8 * 360+180}) {};
     \path [->, thick, bend right=30] (0) edge node[above] {} (1);
     \path [->, thick, bend right=30] (-1) edge node[above] {} (0);
     \path [thick, bend right=20, dashed] (1) edge node[above] {} (2);
     \path [->, thick, bend right=20, dashed] (-2) edge node[above] {} (-1);
     \node[style={circle}] (o) at (4,0) {$f_1$};

    \foreach \x in {0,...,2} {
            \node[style={circle,thick,draw}] (\x) at ({4+2*cos(\x/6 * 360+360/12)},{-3.5+2 * sin(\x/6 * 360+360/12}) {};
        }

    \node[style={circle}] (-1) at ({4+2*cos(-1/6 * 360+360/12)},{-3.5+2 * sin(-1/6 * 360+360/12}) {};
    \node[style={circle}] (3) at ({4+2*cos(3/6 * 360+360/12)},{-3.5+2 * sin(3/6 * 360+360/12}) {};
    \path [->, thick, bend right=20, dashed] (-1) edge node[above] {} (0);
    \path [->, thick, bend right=20] (0) edge node[above] {} (1);
    \path [->, thick, bend right=20] (1) edge node[above] {} (2);
    \path [thick, bend right=20, dashed] (2) edge node[above] {} (3);
    \node[style={circle}] (o) at (4,-3.5) {$N+f_1-1$};

    \foreach \x in {0,...,2} {
            \node[style={circle,thick,draw}] (\x) at ({4+2*cos(\x/6 * 360+360/12+180)},{3.5+2 * sin(\x/6 * 360+360/12+180}) {};
        }

    \node[style={circle}] (-1) at ({4+2*cos(-1/6 * 360+360/12+180)},{3.5+2 * sin(-1/6 * 360+360/12+180}) {};
    \node[style={circle}] (3) at ({4+2*cos(3/6 * 360+360/12+180)},{3.5+2 * sin(3/6 * 360+360/12+180}) {};
    \path [->, thick, bend right=20, dashed] (-1) edge node[above] {} (0);
    \path [->, thick, bend right=20] (0) edge node[above] {} (1);
    \path [->, thick, bend right=20] (1) edge node[above] {} (2);
    \path [thick, bend right=20, dashed] (2) edge node[above] {} (3);
    \node[style={circle}] (o) at (4,3.5) {$N+1$};

     \foreach \x in {-1,...,1} {
            \node[style={circle,thick,draw}] (\x) at ({-0.5-4.5*cos(60)+1.5*cos(\x/4 * 360+60)},{-4.5*sin(60)+1.5 * sin(\x/4 * 360+60}) {};
        }
     \node[style={circle}] (-2) at ({-0.5-4.5*cos(60)+1.5*cos(-3/8 * 360+60)},{-4.5*sin(60)+1.5 * sin(-3/8 * 360+60}) {};
     \node[style={circle}] (2) at ({-0.5-4.5*cos(60)+1.5*cos(3/8 * 360+60)},{-4.5*sin(60)+ 1.5 * sin(3/8 * 360+60}) {};
     \path [->, thick, bend right=30] (0) edge node[above] {} (1);
     \path [->, thick, bend right=30] (-1) edge node[above] {} (0);
     \path [thick, bend right=20, dashed] (1) edge node[above] {} (2);
     \path [->, thick, bend right=20, dashed] (-2) edge node[above] {} (-1);
     \node[style={circle}] (o) at (4,0) {$f_1$};
    \node[style={circle}] (o) at ({-0.5-4.5*cos(60)},{-4.5*sin(60)}) {$f_j$};

     \foreach \x in {-1,...,1} {
            \node[style={circle,thick,draw}] (\x) at ({-0.5-4.5*cos(60)+1.5*cos(\x/4 * 360-60)},{4.5*sin(60)+1.5 * sin(\x/4 * 360-60}) {};
        }
     \node[style={circle}] (-2) at ({-0.5-4.5*cos(60)+1.5*cos(-3/8 * 360-60)},{4.5*sin(60)+1.5 * sin(-3/8 * 360-60}) {};
     \node[style={circle}] (2) at ({-0.5-4.5*cos(60)+1.5*cos(3/8 * 360-60)},{4.5*sin(60)+ 1.5 * sin(3/8 * 360-60}) {};
     \path [->, thick, bend right=30] (0) edge node[above] {} (1);
     \path [->, thick, bend right=30] (-1) edge node[above] {} (0);
     \path [thick, bend right=20, dashed] (1) edge node[above] {} (2);
     \path [->, thick, bend right=20, dashed] (-2) edge node[above] {} (-1);
     \node[style={circle}] (o) at (4,0) {$f_1$};
    \node[style={circle}] (o) at ({-0.5-4.5*cos(60)},{4.5*sin(60)}) {$f_{j'}$};

\end{tikzpicture}
\caption{\label{fig:sft} An SFT construction.}
\end{center}
\end{figure}
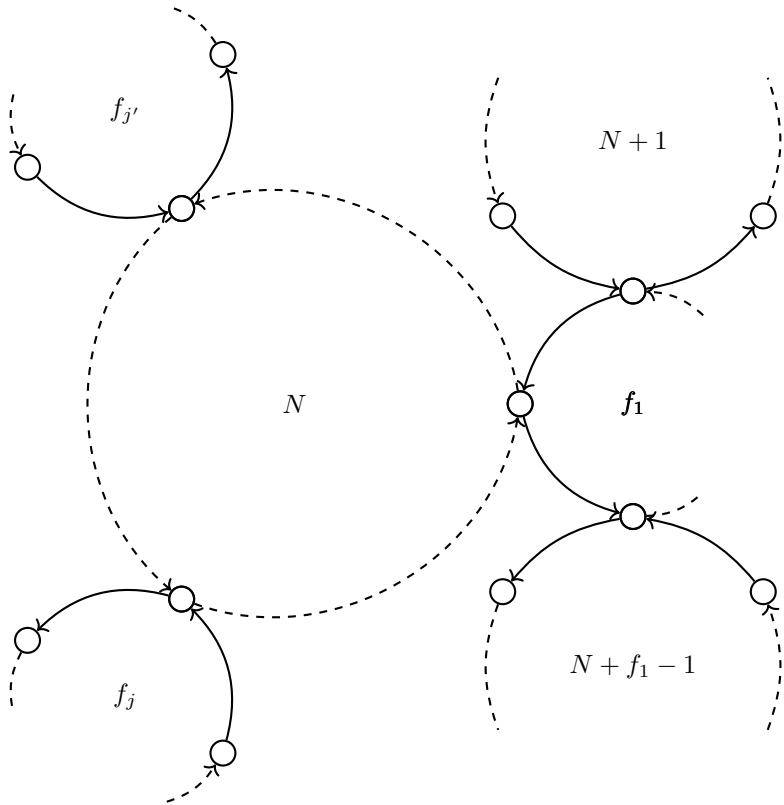

We have just shown the second half of the theorem statement, namely the case when $d=1$.  If $d>1$, it suffices to replace every edge in $G$ with a path of $d$ vertices. Call this new graph $G'$, and note that it is necessarily irreducible. Then for all $n\in \N$ there is a bijection between the points of least period $n$ in $X_G$ and the points of least period $dn$ in $X_G'$. That is, 
$$\operatorname{LPS}(X_{G'})=d\cdot\operatorname{LPS}(X_G),$$
completing the proof.
\end{proof}

Note that the above construction can be used to realize a set of the form $$F\cup \bigcup_{i=1}^nd_iS_i$$
as the LPS of a (generally) reducible SFT. This is done by simply treating each singleton in $F$ and each set $d_iS_i$ as above, and taking the edge shift on the graph consisting of several disconnected components.

\subsection{Sofic shifts}\label{realization-sofic}

At the end of the previous section, we noted how to construct a certain SFT to realize more complicated LPSs than the ones in Theorem \ref{thm-realizeLPSirredSFT}. The cost of restricting ourselves to SFTs is then that the resulting shift is not irreducible. However, below we will realize almost all such sets as the LPS of some irreducible sofic shift on a binary alphabet. 

We will construct a graph $G$ with edge labeling $\Phi$ such that the shift on the labelled graph $\mathcal{G}=(G,\Phi)$, which we denote $X_\mathcal{G}$, has the desired LPS. Since sofic shifts are FP, the following result provides a converse to Theorem \ref{thm-LPSofFP-intro}.

\begin{theorem}\label{thm-realizeLPSsofic}
    Suppose $S$ is a singleton or a set of the form $$F\cup\bigcup_{i=1}^n d_{i}S_{i}$$ where $F$ is finite, each $S_i$ is cofinite, and $n\geq 1$ (i.e. there is at least one cofinite $S_i$). Then there is an irreducible sofic shift with alphabet $\{0,1\}$ whose LPS is $S$.
\end{theorem}

\begin{proof}
    If $S=\{ d \}$, let $w=10^{d-1}$. Then the shift consisting of the orbit of $w^\infty$ is sofic (in fact of finite type), irreducible, and has LPS $S$.

Otherwise, $S$ can be written as 
$$S=F\cup\bigcup_{i=1}^nd_i(\mathbb{N} +k_{i} ),$$ where $k_i$ are such that $$d_i(k_i+1)>2i.$$

Recalling that $\N:=\{1,2,3,...\}$, we begin by writing
$$\bigcap_{i=1}^n d_i(\mathbb{N}+k_i)=d^*(\mathbb{N}+k^*),$$
where
$$d^*=\operatorname{lcm}(\{d_i:1\leq i \leq n\}),\quad k^*=\min\{k\in\N:\forall i, d^*(k+1)\geq d_i(k_i+1)  \}$$

Our construction now goes in several steps. At each step, we will create several individual graphs dubbed \textit{sections} to realize some part of the given LPS. 

\begin{enumerate}
\item  
Consider the set $d_i(\mathbb{N}+k_{i})$ for $i\geq 1$. We will make a section called $\mathcal{T}_i$ to realize most of this set as least periods as follows.

Since we assumed $d_i(k_i+1)>2i$, we also have $d_i(k_i+2)>2i+1$. We set $$u_i=1^{2i}0^{d_i(k_i+1)-2i} \quad \text{and} \quad v_i=1^{2i+1}0^{d_i(k_i+2)-2i-1},$$ so that $|u_i|=d_i(k_i+1),|v_i|=d_i(k_i+2).$

Let $t_i\in\N$ be such that $t_i>1$ and $$t_{i}|u_i|,t_i|v_i|\in d^*(\mathbb{N}+k^*).$$ The exact choice of $t_i$ does not matter, but the latter condition generally requires it to be large.

The section $\mathcal{T}_i$ that we build will be a \textit{torus-section}. We first make a $t_i$ by $t_i$ grid. From each node $(n,m)$ in this grid, add a path to $(n,m+1 \mod t_{i})$ labeled $u_i$ and a path to $(n+1 \mod t_{i},m)$ labeled $v_i$. Within the labelled graph $\mathcal{T}_i$, we can make a bi-infinite path labeled $$(u_i^av_i^b)^\infty$$ for all $a,b\in \mathbb{N}_0=\N\cup\{0\}$. Moreover, any other periodic point in $X_{\mathcal{T}_i}$ will, up to a shift, be of the form $w^\infty$ for $$w=u_i^{n_1}v_i^{m_1}...u_i^{n_k}v_i^{m_k}$$ not purely periodic. The least period of this point is $$\sum_{j=1}^kn_jd_i(k_i+1)+m_jd_i(k_i+2)=ad_i(k_i+1)+bd_i(k_i+2)$$ for some $a,b\in \N_0$. 

Thus, $\operatorname{LPS}(X_{\mathcal{T}_i})$ is $d_i(\mathbb{N}+k_{i})$, \textit{except} for possibly a finite subset. We will treat these missed periods in step 2. 

Figure \ref{fig:torus} is a visual presentation of a torus-section. Two ends of a snaked arrow on the same column or row are in fact a single path, and each arrow labeled $u_i$ or $v_i$ represents a path of sufficiently many edges labeled by the symbols in these words. Call the (0,0) node the \textit{root node}, indicated here by \scalebox{0.6}{$\begin{tikzpicture} \node[style={double,circle,thick,draw}] (A) at (0,0) {};  \end{tikzpicture}$}. 

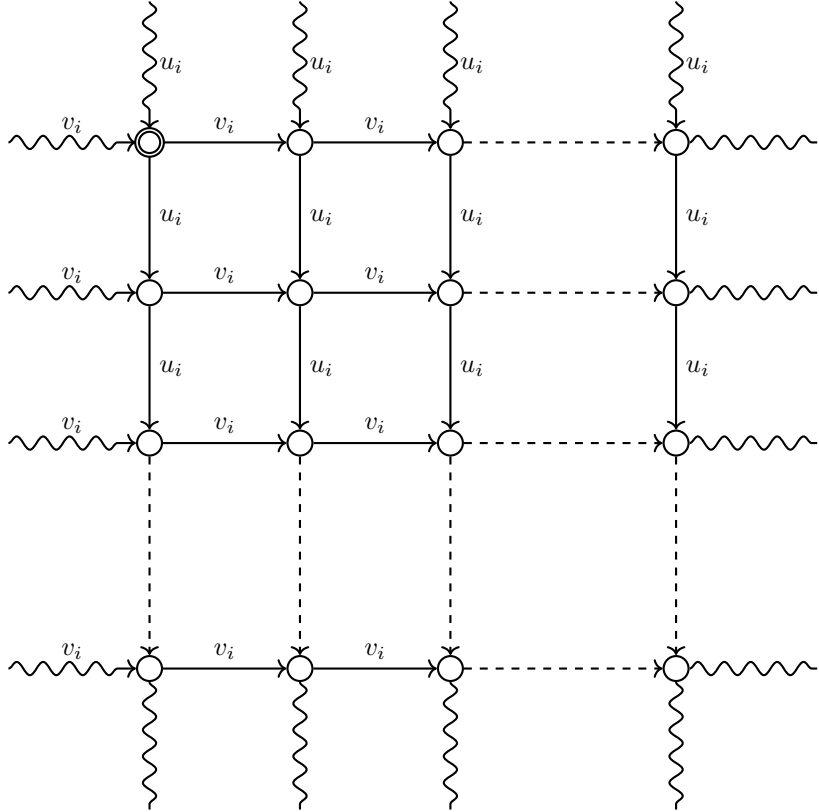
\begin{figure}
\begin{center}
\begin{tikzpicture}

    \node (B0) at (-4,2) {};
    \node (C0) at (-2,2) {};
    \node (D0) at (0,2) {};
    \node (E0) at (3,2) {};

\begin{scope}[every node/.style={circle,thick,draw}]
  
    \node[style=double] (B1) at (-4,0) {};
    \node (C1) at (-2,0) {};
    \node (D1) at (0,0) {};
    \node (E1) at (3,0) {};
\end{scope}
\node (A1) at (-6,0) {};

\node (F1) at (5,0) {};

\draw[->,decorate,decoration={snake,post length = 5pt}, thick] (B0) -- node[right] {$u_i$} (B1);
\draw[->,decorate,decoration={snake,post length = 5pt}, thick] (C0) -- node[right] {$u_i$} (C1);
\draw[->,decorate,decoration={snake,post length = 5pt}, thick] (D0) -- node[right] {$u_i$} (D1);
\draw[->,decorate,decoration={snake,post length = 5pt}, thick] (E0) -- node[right] {$u_i$} (E1);

\draw[->,decorate,decoration={snake,post length = 5pt}, thick] (A1) -- node[above] {$v_i$} (B1);
\path [->,thick] (B1) edge node[above] {$v_i$} (C1);
\path [->,thick] (C1) edge node[above] {$v_i$} (D1);
\path [->,dashed,thick] (D1) edge node[above] {} (E1);

\draw[decorate,decoration={snake}, thick]       (E1) -- (F1);

\begin{scope}[every node/.style={circle,thick,draw}]
  
    \node (B2) at (-4,-2) {};
    \node (C2) at (-2,-2) {};
    \node (D2) at (0,-2) {};
    \node (E2) at (3,-2) {};
\end{scope}
\node (A2) at (-6,-2) {};

\node (F2) at (5,-2) {};

\draw[->,decorate,decoration={snake,post length = 5pt}, thick] (A2) -- node[above] {$v_i$} (B2);
\path [->,thick] (B2) edge node[above] {$v_i$} (C2);
\path [->,thick] (C2) edge node[above] {$v_i$} (D2);
\path [->,dashed,thick] (D2) edge node[above] {} (E2);

\draw[decorate,decoration={snake}, thick]       (E2) -- (F2);

\path [->,thick] (B1) edge node[right] {$u_i$} (B2);
\path [->,thick] (C1) edge node[right] {$u_i$} (C2);
\path [->,thick] (D1) edge node[right] {$u_i$} (D2);
\path [->,thick] (E1) edge node[right] {$u_i$} (E2);

\begin{scope}[every node/.style={circle,thick,draw}]
  
    \node (B3) at (-4,-4) {};
    \node (C3) at (-2,-4) {};
    \node (D3) at (0,-4) {};
    \node (E3) at (3,-4) {};
\end{scope}
\node (A3) at (-6,-4) {};

\node (F3) at (5,-4) {};

\draw[->,decorate,decoration={snake,post length = 5pt}, thick] (A3) -- node[above] {$v_i$} (B3);
\path [->,thick] (B3) edge node[above] {$v_i$} (C3);
\path [->,thick] (C3) edge node[above] {$v_i$} (D3);
\path [->,dashed,thick] (D3) edge node[above] {} (E3);

\draw[decorate,decoration={snake}, thick]       (E3) -- (F3);

\path [->,thick] (B2) edge node[right] {$u_i$} (B3);
\path [->,thick] (C2) edge node[right] {$u_i$} (C3);
\path [->,thick] (D2) edge node[right] {$u_i$} (D3);
\path [->,thick] (E2) edge node[right] {$u_i$} (E3);

\begin{scope}[every node/.style={circle,thick,draw}]
  
    \node (B4) at (-4,-7) {};
    \node (C4) at (-2,-7) {};
    \node (D4) at (0,-7) {};
    \node (E4) at (3,-7) {};
\end{scope}
\node (A4) at (-6,-7) {};

\node (F4) at (5,-7) {};

\draw[->,decorate,decoration={snake,post length = 5pt}, thick] (A4) -- node[above] {$v_i$} (B4);
\path [->,thick] (B4) edge node[above] {$v_i$} (C4);
\path [->,thick] (C4) edge node[above] {$v_i$} (D4);
\path [->,dashed,thick] (D4) edge node[above] {} (E4);

\draw[decorate,decoration={snake}, thick]       (E4) -- (F4);

\path [->,thick, dashed] (B3) edge (B4);
\path [->,thick, dashed] (C3) edge (C4);
\path [->,thick, dashed] (D3) edge (D4);
\path [->,thick, dashed] (E3) edge (E4);


    \node (B5) at (-4,-9) {};
    \node (C5) at (-2,-9) {};
    \node (D5) at (0,-9) {};
    \node (E5) at (3,-9) {};
\draw[decorate,decoration={snake}, thick]   (B4) -- (B5);
\draw[decorate,decoration={snake}, thick]  (C4) -- (C5);
\draw[decorate,decoration={snake}, thick]   (D4) -- (D5);
\draw[decorate,decoration={snake}, thick]   (E4) -- (E5);

\end{tikzpicture}
\caption{\label{fig:torus} A torus-section.}
\end{center}
\end{figure}

\item Next, write $F':= S\setminus \bigcup_{i=1}^n \operatorname{LPS}(X_{\mathcal{T}_i})$, and note that $F'$ is finite. We will realize each $j\in F'$ by making a labeled \textit{cycle-section} $\mathcal{C}_j$. First, if $j\neq 1$, construct a simple cycle of length $c_j\cdot j$, where $c_j$ is such that $c_j\cdot j\in d^*(\mathbb{N}+k^*)$ and $c_j>1$. Label this cycle with the word $$w_j=10^{j-1}$$ repeated $c_j$ times. 

If $j=1$, make a simple cycle of length $c_1$ such that $c_1\in d^*(\mathbb{N}+k^*)$ and $c_1>2n+1$. Label this cycle with the word $w_1=1$ repeated $c_1$ times. Thus, for all $j\in F'$, we have $\operatorname{LPS}(X_{\mathcal{C}_j})=\{j\}$. 

Figure \ref{fig:cycle} is a visual representation of such a cycle-section, with the same visual indicators as in step 1. Again, call the 0 node the root node, indicated by \scalebox{0.6}{$\begin{tikzpicture} \node[style={double,circle,thick,draw}] (A) at (0,0) {};  \end{tikzpicture}$}. 

\begin{figure}
\begin{center}
\begin{tikzpicture}
\begin{scope}[every node/.style={circle,thick,draw}]
  
    \node[style=double] (B) at (-4,0) {};
    \node (C) at (-2,0) {};
    \node (D) at (0,0) {};
    \node (E) at (3,0) {};
\end{scope}
\node (A) at (-6,0) {};

\node (F) at (5,0) {};

\draw[->,decorate,decoration={snake,post length = 5pt}, thick] (A) -- node[above] {$w_j$} (B);
\path [->,thick] (B) edge node[above] {$w_j$} (C);
\path [->,thick] (C) edge node[above] {$w_j$} (D);
\path [->,dashed,thick] (D) edge node[above] {} (E);

\draw[decorate,decoration={snake}, thick]       (E) -- (F);
\end{tikzpicture}
\caption{\label{fig:cycle} A cycle-section.}
\end{center}
\end{figure}
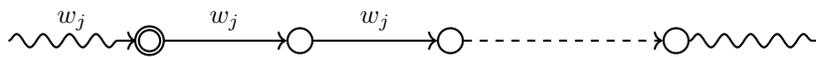

\item Finally, identify the root nodes \scalebox{0.6}{$\begin{tikzpicture} \node[style={double,circle,thick,draw}] (A) at (0,0) {};  \end{tikzpicture}$} of each section together, so that the graph is irreducible. Call this labeled graph $\mathcal{G}$.
\end{enumerate}

We now have an irreducible sofic shift $X_\mathcal{G}$. By construction, $$S=\bigcup_{j\in F'} \operatorname{LPS}(X_{\mathcal{C}_j}) \cup \bigcup_{i=1}^n\operatorname{LPS}(X_{\mathcal{T}_i}) \subset \operatorname{LPS}(X_\mathcal{G}).$$ It remains to show that no other least periods were created in joining the sections together. 

Our construction is such that the shifts on all $\mathcal{T}_i$ and $\mathcal{C}_j$ are disjoint. Having taken $t_i,c_j>1$, we also have the following property for bi-infinite labeled paths $\tau$ on $\mathcal{G}$. \begin{enumerate}
    \item $\tau$ contains $1^{c_1}$ if and only if it passes through $\mathcal{C}_1$,
    \item $\tau$ contains the word $010^j1$ if and only if it passes through $\mathcal{C}_j$, $j\neq 1$, and
    \item $\tau$ contains $01^{2i}0$ or $01^{2i+1}0$ if and only if it passes through $\mathcal{T}_i$.
\end{enumerate} Moreover, a path that goes from section to another must pass through the root node. 

Let $\tau$ be a bi-infinite path in $G$ that presents a periodic point $x\in X_\mathcal{G}$. If $\tau$ only passes through a single section, the least period of $x$ is in $S$ by construction.

Otherwise, the above allows us to write, up to a shift, $x=w^\infty$ with $$w=\overline{w}_{1}\overline{w}_{2}\cdots\overline{w}_{m},$$ with $w$ not purely periodic and where each $\overline{w}_{k}$ is presented by a path in a single section, starting and ending at the root node.

Writing $x$ in this way, the least period of $x$ is
$$\sum_{k=1}^m |\overline{w}_{k}|,$$
which will be shown to be in $d^*(\N+k^*)$, and hence in $S$.
Since $d^*(\N+k^*)$ is closed under addition, it suffices to show that 
$$|\overline{w}_{k}|\in d^*(\N+k^*)$$
for all $k$.

Suppose $\overline{w}_{k}$ traverses a torus-section $\mathcal{T}_{i_k}$. Then $\overline{w}_{k}$ must consist of a combination of the words $u_{i_k},v_{i_k}$ such that each word appears some integer multiple of $t_{i_k}$ times. In other words, $$|\overline{w}_{k}|=a\cdot t_{i_k}|u_{i_k}|+b\cdot t_{i_k}|v_{i_k}|$$
for some $a,b\in \N_0$, not both 0. Recalling again that $t_{i_k}|v_{i_k}|,t_{i_k}|u_{i_k}|\in d^*(\N+k^*)$, and that this latter set is closed under addition, we get that
$$|\overline{w}_{k}|\in d^*(\N+k^*).$$

If $\overline{w}_k$ traverses a cycle-section $\mathcal{C}_{j_k}$, the argument is similar. Indeed, since the path presenting $\overline{w}_{k}$ starts and ends at the root node, the length of this path must be an integer multiple of the cycle length of $\mathcal{C}_{j_k}$. That is,
$$|\overline{w}_{k}|=a'\cdot c_{j_k} \cdot j_k \in d^*(\N+k^*),$$
for some $a'\in \N$.
Having shown that $|\overline{w}_{k}|\in d^*(\N+k^*)$ for all $j$, we conclude that the least period of $x$ is in $d^*(\N+k^*)\subset S$, as desired. 

We have thus shown that $$\operatorname{LPS}(X_\mathcal{G})=S,$$ completing the proof.
\end{proof}

\subsection{Gap shifts}\label{sec-gapconstruction}

The inherent structure of gap shifts allows us to realize their LPS with relative ease. In fact, we obtain the desired LPS via a \textit{sofic} gap shift. The majority of the proof below follows from Proposition \ref{prop-almostsumprops} and our observations during the proof of Theorem \ref{thm-LPSofgap}.

\begin{theorem}\label{thm-realizeLPSgap}
    Let $Q$ be a set as outlined in the conclusion of Theorem \ref{thm-LPSofgap}. Then there exists a sofic $S$-gap shift whose LPS is $Q$. Moreover, if $Q$ is almost sum-closed, then there exists an $S$-gap SFT with least period set $Q$.
\end{theorem}

\begin{proof}
    Let $Q$ be a set as indicated and for $S\subset\N_0$, let $X(S)$ be the $S$-gap shift. We will first prove the second half of the theorem. If $Q$ is a singleton $\{ n \}$, we take $S=\{ n-1 \}$. Then $X(S)$ is a finite orbit of length $n$. Singletons are trivially almost sum-closed, and $X(S)$ is an SFT, so we are done.

    Next, suppose $Q$ is an almost sum-closed set $dR$ with $R$ cofinite. We will find a finite set of elements whose almost sum-closure is $Q$. Write first $$Q=F\cup \{dN,d(N+1),d(N+2),...\},$$
    where every element of $F$ is divisible by $d$ and strictly smaller than $dN$. Now by assumption $Q$ is almost sum-closed, so for some $F'\subset \N$, $$\Sigma_F^*=F\cup F'\subset Q.$$
    Moreover, if we set $$B=\{dN,d(N+1),d(N+2),...,d\cdot2N\},$$
    then $$\Sigma^*_B=\{dN,d(N+1),d(N+2),...\}.$$
    Thus, in this particular case, $$\Sigma^*_{F\cup B}=F\cup \{dN,d(N+1),d(N+2),...\}=Q.$$
    We noted in the proof of Theorem \ref{thm-LPSofgap} that if $S$ is finite, 
    $$\operatorname{LPS}(X(S))=\Sigma^*_{S+1}.$$
    Thus, if we set $S=(F\cup B) - 1$, then $X(S)$ is an SFT by Proposition \ref{prop-classifySgaps}, and $$\operatorname{LPS}(X(S))=Q.$$
    
    Finally, suppose $$Q= {dR} \cup {\{ 1 \}}$$ with $R$ cofinite and $dR$ almost sum-closed. Then $dR$ is eventually periodic, so $dR-1$ is also eventually periodic. Thus, setting $S=dR-1$, the $S$-gap shift $X(S)$ is sofic by Proposition \ref{prop-classifySgaps}. 
    Now, we recall from the proof of Theorem \ref{thm-LPSofgap} that if $S$ is infinite, 
    $$\operatorname{LPS}(X(S))=\Sigma_{S+1}^*\cup\{1\}.$$
    But by Proposition \ref{prop-almostsumprops}, we have $$\Sigma_{(dR-1)+1}^*=\Sigma_{dR}^*=dR,$$
    and hence for $S=dR-1$,
    $$\operatorname{LPS}(X(S))=dR\cup\{1\},$$
    completing the proof.
\end{proof}

\subsection{Arbitrary LPS}

We conclude our collection of realizations by showing that \textit{any} subset of the natural numbers is the LPS of some subshift of the full 2-shift. 

\begin{theorem}\label{thm-realizeLPSarbitrary}
    Let $S\subset \mathbb{N}$ be arbitrary. Then there exists a shift space $X \subset \{ 0,1 \}^\mathbb{Z}$  whose LPS is $S$.
\end{theorem}

\begin{proof}
    If $S$ is finite, take  $X=\{\sigma^n((10^{k-1})^\infty):k\in S,n\in \mathbb{Z}\}$. Then $X$ is a finite shift space.

    Suppose $S=\{k_{1},k_{2},\dots\}$ is infinite and written in strictly increasing order. We consider two cases.

    If $1\in S$, for all $i\in \mathbb{N}$, take $u_{i}=10^{k_{i}-1}$, which is not purely periodic. Then the LPS of
    $$
    X=\overline{\{ \sigma^n(u_{i}^\infty):i \in \mathbb{N}, n\in \mathbb{Z} \}}
    $$
    certainly contains $S$. We will show that it contains no other periods. 

    Let $X$ be as above, and take $x\in X$ periodic with least period $k$, where $x=w^\infty$ for $w$ not purely periodic. If there are no 1s in $w$, then $x=0^\infty$, which has least period 1. This period is in $S$ by assumption, and we are safe.

    Otherwise, take $x^{(n)}$ to be a sequence of points such that for all $n\in\N$, $$x^{(n)}=\sigma^j(u_i^\infty)$$
    for some $i\in\N,j\in\Z$ and
    $$\lim_{n\to\infty} x^{(n)} = x.$$
    Now there exists an $N$ such that for all $n\geq N  $, we have that $$x^{(n)}_{[-k,k-1]} = x_{[-k,k-1]} =w^2.$$ 
    Since there is a 1 somewhere in $w$, we know that the word $10^{l}1$ is in $w^2$ for some $l$. But if this word is contained in $x^{(n)}$, then $x^{(n)}$ must be some shift of $(10^l)^\infty$, and therefore $l=k_i-1$ for some $k_i\in S$. We conclude that for some $i\in \N,j\in\Z$. 
    $$x^{(n)}=\sigma^j(u_i^\infty)$$
    for all $n\geq N$. Hence, $$x=\sigma^j(u_i^\infty),$$
    and the least period of $x$ is in $S$ by construction.

    Suppose now that $1\notin S$. Set $u_1=10^{k_1-1}$. For $i>1$, take $$u_{i}=1u_{1}^{g_i}u_{i}',$$ where $g_i\in\N_0$ and $u_{i}'$ is a strict prefix of $u_1$ such that $|u_1^{g_i}u_{i}'|=k_i-1$. This way $u_i$ is not purely periodic, since the subword 110 only occurs at the very start. At the same time, $u_i$ contains the subword 110 for any $i>1$, since $k_1\neq1.$
    
    Now the LPS of the shift space
    $$
    X=\overline{\{ \sigma^n(u_{i}^\infty):i \in \mathbb{N}, n\in \mathbb{Z} \}}
    $$
    again contains $S$. We show that no other least periods are created, using a similar, but slightly more subtle, argument as in the previous case.

    Let again $X$ be as above, and take $x\in X$ periodic with least period $k$, where $x=w^\infty$ for $w$ not purely periodic. 
    
    Moreover, take once again $x^{(n)}$ to be a sequence of points such that for all $n\in\N$, $$x^{(n)}=\sigma^j(u_i^\infty)$$
    for some $i\in\N,j\in\Z$ and
    $$\lim_{n\to\infty} x^{(n)} = x.$$

    Slightly altering our previous argument, there exists an $N$ such that for all $n\geq N  $, we have that $$x^{(n)}_{[-k,2k-1]} = x_{[-k,2k-1]} =w^3.$$ 

    Suppose that the word 110 appears in $x$. Then $w^3$ contains a word of the form $110\tilde{u}110$ for some subword $\tilde{u}$. For all $n$ large enough, $x^{(n)}$ then also contains $110\tilde{u}110$. This is only possible if $$x^{(n)}=\sigma^j(u_i^\infty)$$ for some $i\in\N,j\in\Z$ for all large $n$, and hence $x=\sigma^j(u_i^\infty)$ with least period $k_i\in S$.

    If the word 110 does \textit{not} appear in $x$, take a subsequence $n_m$ such that
    $$x^{(n_m)}_{[-mk,mk-1]} = x_{[-mk,mk-1]}=w^{2m}$$
    for all $m\in\N$. Since the word 110 does not appear in $w^{2m}$ for any $m$, we know that if 110 appears in $x^{(n_m)},$ it must be outside of the window $[-mk,mk-1]$. In other words, for all $m$,
    $$x^{(n_m)}_{[-mk,mk-1]}=\sigma^j(u_1^\infty)_{[-mk,mk-1]}$$ for some $j\in\Z$.
    We then conclude that $x=\sigma^j(u_1^\infty)$, which has least period $k_1\in S$.

\end{proof}

\bibliographystyle{alpha}
\bibliography{refs}

\end{document}